\documentclass[letterpaper, 12pt]{article}

\usepackage[margin=1in]{geometry}
\usepackage{amsmath,amssymb,amsthm, amsfonts}
\usepackage{mathtools}

\usepackage[colorlinks=true, linkcolor=teal, citecolor=-teal,
pagebackref=true]{hyperref}
\usepackage{tikz-cd}
\usepackage{enumitem}

\usepackage[reftex]{theoremref}

\usepackage{charter}
\usepackage{newtxmath}

\newtheorem{theorem}{Theorem}[section]
\newtheorem*{theorem*}{Theorem}
\newtheorem{lemma}[theorem]{Lemma}
\newtheorem{proposition}[theorem]{Proposition}

\newtheorem{conjecture}[theorem]{Conjecture}
\newtheorem{claim}[theorem]{Claim}

\theoremstyle{definition}
\newtheorem{example}{Example}[section]
\newtheorem*{example*}{Example}

\theoremstyle{remark}
\newtheorem*{remark*}{Remark}

\theoremstyle{definition}
\newtheorem{definition}{Definition}[section]
\newtheorem*{definition*}{Definition}

\def\eps{{\varepsilon}}
\def\ga{{\alpha}}
\def\bga{{\boldsymbol \alpha}}
\def\gd{{\delta}}
\def\gk{{\kappa}}

\newcommand*\diff{\mathop{}\!\mathrm{d}}

\newcommand{\calA}{\mathcal A}

\newcommand{\calE}{\mathcal E}

\newcommand{\calK}{\mathcal K}

\newcommand{\p}{\mathbf p}
\newcommand{\q}{\mathbf q}

\newcommand{\ba}{\mathbf a}
\newcommand{\bb}{\mathbf b}

\newcommand{\br}{\mathbf r}
\newcommand{\bs}{\mathbf s}
\newcommand{\bv}{\mathbf v}
\newcommand{\bw}{\mathbf w}

\newcommand{\X}{\mathbf X}
\newcommand{\Y}{\mathbf Y}
\newcommand{\Z}{\mathbf Z}
\newcommand{\y}{\mathbf y}
\newcommand{\NN}{\mathbb N}
\newcommand{\PP}{\mathbb P}
\newcommand{\RR}{\mathbb R}
\newcommand{\TT}{\mathbb T}
\newcommand{\ZZ}{\mathbb Z}
\newcommand{\Zpos}{\mathbb Z_+}

\newcommand{\Rpos}{\mathbb R_+}
\newcommand{\bone}{\mathbf 1}
\newcommand{\bzero}{\mathbf 0}

\renewcommand{\leq}{\leqslant}
\renewcommand{\geq}{\geqslant}
\renewcommand{\subset}{\subseteq}
\renewcommand{\supset}{\supseteq}

\DeclarePairedDelimiter{\abs}{\lvert}{\rvert}
\DeclarePairedDelimiter{\parens}{\lparen}{\rparen}

\DeclarePairedDelimiter{\set}{\lbrace}{\rbrace}
\DeclarePairedDelimiter{\norm}{\lVert}{\rVert}

\DeclareMathOperator{\Leb}{m}

\DeclareMathOperator{\lcm}{lcm}

\title{Metric bootstraps for limsup sets}

\author{Felipe A.~Ram{\'i}rez \\ Wesleyan University }

\makeatletter
\@namedef{subjclassname@2020}{\textup{2020} Mathematics Subject Classification}
\makeatother

\date{\footnotesize{\it In admiration of Khintchine, whose theorem has
    inspired a century's worth of research}}

\begin{document}

\frenchspacing

\maketitle
\thispagestyle{empty}

\begin{abstract}
  In metric Diophantine approximation, one frequently encounters the
  problem of showing that a limsup set has positive or full
  measure. Often it is a set of points in $m$-dimensional Euclidean
  space, or a set of $n$-by-$m$ systems of linear forms, satisfying
  some approximation condition infinitely often. The main results of
  this paper are bootstraps: if one can establish positive measure for
  such a limsup set in $m$-dimensional Euclidean space, then one can
  establish positive or full measure for an associated limsup set in
  the setting of $n$-by-$m$ systems of linear forms.  Consequently, a
  class of $m$-dimensional results in Diophantine approximation can be
  bootstrapped to corresponding $n$-by-$m$-dimensional results. This
  leads to short proofs of existing, new, and hypothetical theorems
  for limsup sets that arise in the theory of systems of linear
  forms. We present several of these.
\end{abstract}

\tableofcontents

\section*{Notation and conventions}

\begin{itemize}
\item $\NN$ denotes the positive integers.
  
\item $\Zpos$ and $\Rpos$ denote the nonnegative integers and the
  nonnnegative reals.
  
\item $\Leb$ denotes Lebesgue measure, possibly with a subscript when the
  dimension is not clear from the context.
  
\item $\abs{\cdot}$ denotes the maximum norm, with respect to which
  all distances are to be understood.
\end{itemize}
\section{Introduction and results}

\subsection{Setting}
\label{sec:setting-background}

Consider the collection of $n$-by-$m$ systems of linear forms that
send infinitely many integer vectors into some prescribed union of
target balls in $\RR^m$, that is, the set of matrices
$\X \in M_{n\times m}(\RR)$ for which there are infinitely many
$\q\in\ZZ^n$ such that
\begin{equation}\label{eq:one}
  \q\X  \in B(\bzero, \psi(\q)) + P(\q),
\end{equation}
where $\psi:\ZZ^n\to\Rpos$ is a given function,
$B(\bzero, \psi(\q))\subset\RR^m$ is the ball of radius
$\psi(\q)\geq 0$ centered at the origin $\bzero\in\RR^m$, and
$P(\q)\subset \RR^m$ is a discrete set determining which translates of
the ball are acceptable targets. That set, which is denoted here by
\begin{equation*}
  W_{n,m}^P(\psi) = \set*{\X \in M_{n\times m}(\RR) : \textrm{ (\ref{eq:one}) holds for infinitely many }\q \in \ZZ^n},
\end{equation*}
is at the heart of \emph{asymptotic metric Diophantine
  approximation}. Many problems in the field concern sets of the form
$W_{n,m}^P(\psi)$ for various choices of $P$. For classical
\emph{homogeneous approximation} one has $P(\q)$ a subset of $\ZZ^m$,
while for \emph{inhomogeneous approximation} $P(\q)$ is a subset of
$\y + \ZZ^m$ where $\y\in\RR^m$ is a fixed inhomogeneous
parameter. Often, the definition of $P$ involves coprimality or
congruence conditions on the entries of its members.

The main question addressed here is \emph{what is the Lebesgue measure
  of $W_{n,m}^P(\psi)$?} It should be noted that $W_{n,m}^P(\psi)$ is
the limsup set corresponding to the sets
\begin{equation*}
  A(\q) = \set*{\X \in M_{n\times m}(\RR) : \textrm{ (\ref{eq:one}) holds}}.
\end{equation*}
As such, there are cases where the question is answered immediately by
the Borel--Cantelli lemma. Namely, if for every ball
$B\subset M_{n\times m}(\RR)$, the series $\sum_\q \Leb (A(\q)\cap B)$
converges, then by the Borel--Cantelli lemma, $W_{n,m}^P(\psi)$ has
zero measure. 

The focus, therefore, is on situations where the series diverges, and
the goal is always to show that $W_{n,m}^P(\psi)$ has positive or full
measure. Our main result is that if that goal can be achieved in the
$1$-by-$m$ setting, then it can also be achieved in the $n$-by-$m$
setting for all $n\geq 1$.

\subsection{Metric dichotomies}
\label{sec:metric-dichotomies}

The prototypical result for sets of the form $W_{n,m}^P(\psi)$ is a
zero-full metric dichotomy: an assertion that the measure is either
zero or full, together with a criterion to determine which of the two
it is. It is usually determined by the convergence or divergence of a
series associated to $\psi$ and $P$---a series which is understood as
a proxy for the measure sum mentioned above.

The seminal example is Khintchine's theorem (1924,
1926,~\cite{Khintchineonedimensional,Khintchine}). It says that if
$\psi:\NN\to\Rpos$ is nonincreasing, and $P(q) = \ZZ^m$ for all
$q\in\NN$, then $\Leb(W_{1,m}^P(\psi))$ is zero or full depending on
the convergence or divergence of the series $\sum_{q} \psi(q)^m$. The
$n$-by-$m$ version of Khintchine's theorem is known as the
Khintchine--Groshev Theorem (1938,~\cite{Groshev}). It says, for
functions $\psi:\ZZ^n\to\Rpos$ only depending on $\abs{\q}$ and
satisfying a monotonicity condition, and with $P(\q) = \ZZ^m$ as
before, that $\Leb(W_{n,m}^P(\psi))$ is zero or full depending on the
convergence or divergence of the series $\sum_{q}q^{n-1}\psi(q)^m$.

An inhomogeneous version of the Khintchine--Groshev theorem can be
found in the work of Sprindzuk (1979,~\cite{Sprindzuk}). The statement
is exactly that of the (homogeneous) Khintchine--Groshev theorem,
except that $P(\q) = \y + \ZZ^m$ for every $\q\in\ZZ^n$, where
$\y\in\RR^m$ is a fixed inhomogeneous parameter. The $1$-by-$1$ case
had been proved in 1958 by Sz{\"u}sz~\cite{SzuszinhomKT} and the
$1$-by-$m$ case had been proved in 1964 by Schmidt~\cite{Schmidt};
these constitute the inhomogeneous Khintchine theorem.

Note that when $P(\q) = \y+\ZZ^m$, all integer translates of the ball
$B(\y, \psi(\q))$ are admissible targets for $\q\X$. Other classical
settings impose restrictions involving coprimality, such as in the
Duffin--Schaeffer conjecture (1941,~\cite{duffinschaeffer}), now a
theorem due to Koukoulopoulos--Maynard (2020,~\cite{KMDS}). It says
that if $\psi:\NN\to\Rpos$ and $P(q) = \set{p\in\ZZ : \gcd(p,q)=1}$
for all $q\in\NN$, then $\Leb(W_{1,1}^P(\psi))$ is zero or full
depending on the convergence or divergence of the series
$\sum_{q} \varphi(q)\psi(q)/q$, where $\varphi$ is Euler's totient
function. The $m$-dimensional ($m>1$) version of this was proved by
Pollington and Vaughan in 1990~\cite{PollingtonVaughan} and the
corresponding $n$-by-$m$ version was proved recently by
the author~\cite{ramirez2023duffinschaeffer}, after having been
conjectured by Beresnevich, Bernik, Dodson, and Velani in
2009~\cite{BBDV}.

There are many other such dichotomies. For example, the monotonicity
assumption appearing in the Khintchine--Groshev theorem has been the
subject of much research, and versions where that assumption is
relaxed have appeared in~\cite{BVKG,Sprindzuk}. In 2010,
Beresnevich--Velani proved a multivariate---meaning $\psi(\q)$ does
not have to depend only on $\abs{\q}$---version with no monotonicity
assumption whenever $m>1$~\cite{BVKG}. The $n=1$ cases of this had
been established by Gallagher in 1965~\cite{Gallagherkt}. In 2023
Allen and the author proved a version of the (univariate)
inhomogeneous Khintchine--Groshev theorem without monotonicity
assumptions whenever $nm>2$~\cite{allen2021independence}. The $n=1$
cases of this had been proved by Yu in 2019~\cite{Yu}.

\subsection{Main results}

The purpose of this article is to show that for metric dichotomies,
the $1$-by-$m$ theory underpins the $n$-by-$m$ theory: the $n$-by-$m$
theorems in the previous section follow from their $1$-by-$m$ cases;
several $n$-by-$m$ theorems that do not already appear in the
literature follow from $1$-by-$m$ theorems that do; and certain open
conjectures in the $1$-by-$m$ setting give rise to their $n$-by-$m$
counterparts.

All of this follows from applying our main results,
\th\ref{thm:bootstrap,thm:bootstrap11}. Briefly, these state that for
a given function $\psi:\Zpos^n\to\Rpos$, there is an associated family
of functions $\Psi_Q:\NN \to [0,\infty]$ ($Q\geq 1$) such that if
$W_{1,m}^P(\Psi_1)$ has positive measure then so does
$W_{n,m}^P(\psi)$, and if $\Leb\parens{W_{1,m}^P(\Psi_Q)}$ is bounded away from $0$
for $Q\geq 1$ then $W_{n,m}^P(\psi)$ has full measure.

\

\th\ref{thm:bootstrap,thm:bootstrap11} apply when $P$ conforms to the
following definition.

\begin{definition}\th\label{def:well}
  For each $d\geq 1$, let $P(d)\subset(\RR/d\ZZ)^m$ be a nonempty
  finite set, and define for each $\q\in\Zpos^n$ the set $P(\q)$ to be
  the lift of $P(\gcd(\q))$ to $\RR^m$. We say the sequence of sets
  $P:= (P(d))_{d\in\NN}$ is \emph{relatively uniformly discrete} if
  there exist positive real numbers $b, c>0$ such that for every
  $d\geq 1$ there is a subset $P'(d) \subset P(d)$ with
  \begin{equation}\th\label{eq:udiscrete}
    \#P'(d) \geq c \#P(d) \qquad\textrm{ and }\qquad b \leq \abs{\p_1
      - \p_2} \qquad\textrm{ for all distinct }  \p_1, \p_2\in P'(\q),
  \end{equation}
  where $\q\in\ZZ^n$ is any vector with $\gcd(\q)=d$ and $P'(\q)$
  denotes the lift of $P'(d)$ to $\RR^m$. If, moreover, there exists
  $a>0$ such that for all $d\geq 1$,
  \begin{equation}\label{eq:well-spread}
    \frac{ad}{(\#P(d))^{1/m}} \leq \abs{\p_1 - \p_2} \qquad\textrm{ for all distinct }\p_1, \p_2\in P'(\q)\quad (\gcd(\q)=d),
  \end{equation}
  then we will say $P$ is \emph{relatively well-spread}. If these
  conditions can be accomplished with $c=1$, then we drop `relatively'
  from the terminology.
\end{definition}

\begin{remark*}
  Well-spreadedness means that it is possible to occupy a fixed
  positive proportion of $(\RR/d\ZZ)^m$ by placing disjoint balls
  around the points in $P(d)$. The condition of relative
  well-spreadedness is satisfied in all the metric dichotomies
  discussed in \S\ref{sec:metric-dichotomies}.
\end{remark*}

Notice that if $P(\q)$ is the lift to $\RR^m$ of a finite subset of
$(\RR/\gcd(\q)\ZZ)^m$, then~(\ref{eq:one}) is a $\ZZ^{mn}$-periodic
condition on $\X$. If this is true of each $\q\in\ZZ^n$, then
$W_{n,m}^P(\psi)$ is a $\ZZ^{mn}$-periodic set. In this case, it is
convenient to redefine it as
\begin{equation*}
  W_{n,m}^P(\psi) = \set*{\X \in \TT^{nm}\subset M_{n\times m}(\RR) : \textrm{ (\ref{eq:one}) holds for infinitely many }\q \in \ZZ^n},
\end{equation*}
where $\TT^{nm}$ is identified with the elements of
$M_{n\times m}(\RR)$ whose entries lie in $[0,1]$. The main theorems
are stated with respect to this view of $W_{n,m}^P(\psi)$.

The last definition to be introduced before stating the results is
that of the auxiliary functions $\Psi_Q$. Given $m,n\in\NN$,
$\psi:\ZZ^n\to\Rpos$, and $Q\geq 1$, define
$\Psi_Q:\NN\to[0,\infty]$ by
\begin{equation}\label{eq:Psidef}
  \Psi_Q(d) =\parens*{\sum_{\substack{\q\in\ZZ^n \\ \gcd(\q)=d \\ \abs*{\q/\gcd(\q)} \geq Q}}\psi(\q)^m}^{1/m}.
\end{equation}
Let $\Psi:=\Psi_1$.

\begin{theorem}\th\label{thm:bootstrap}
  Fix $m,n \in\NN$.
  \begin{enumerate}[label = \alph*)]
  \item If $P$ is relatively well-spread, then for functions
    $\psi:\Zpos^n\to\Rpos$,
    \begin{equation*}
      \Leb(W_{1,m}^P(\Psi)) > 0 \qquad\implies\qquad \Leb(W_{n,m}^P(\psi)) > 0.
    \end{equation*}
    
  \item If $P$ is relatively uniformly discrete, then for bounded
    functions $\psi:\Zpos^n\to\Rpos$,
    \begin{equation*}
      \Leb(W_{1,m}^P(\Psi)) > 0 \qquad\implies\qquad  \Leb(W_{n,m}^P(\psi)) >0.
    \end{equation*}
  \end{enumerate}
\end{theorem}

\begin{remark*}
  The requirement that the approximating function $\psi$ be supported
  in the positive orthant of $\ZZ^n$ is a technical convenience. In
  the applications presented in~\S\ref{sec:applications}, one can
  always assume without loss of generality that a divergence condition
  is met in that orthant.
\end{remark*}

\begin{remark*}
  It is possible that $\Psi(d) = \infty$ for some $d\geq 1$. We will
  see that in such cases one automatically has
  $\Leb(W_{n,m}^P(\psi))=1$, so there is nothing to prove
  (\th\ref{prop:Psifinite}). In the meantime, let it be understood
  that a ball of infinite radius contains all of $\RR^m$. Notice,
  also, that if $n=1$, then $\psi(q) = \Psi(q)$ for all $q\geq 1$, and
  the theorem is trivial.
\end{remark*}

One might suspect that if $W_{1,m}^P(\Psi)$ has full measure rather
than just positive measure, then it is possible to conclude that
$W_{n,m}^P(\psi)$ also has full measure. However,
\th~\ref{ex:fulltopos} in~\S\ref{sec:examples} shows that this is not
true in general.

The next result achieves full measure for $W_{n,m}^P(\psi)$ from
positive measure of $W_{1,m}^P(\Psi_Q)$ for all $Q\geq 1$.

\begin{theorem}\th\label{thm:bootstrap11}
  Fix $m,n \in\NN$ and assume $\lim_{d\to\infty}\#P(d) = \infty$.
  \begin{enumerate}[label = \alph*)]
  \item If $P$ is relatively well-spread, then for functions
    $\psi:\Zpos^n\to\Rpos$,
    \begin{equation*}
      \inf_{Q}\Leb(W_{1,m}^P(\Psi_Q)) > 0 \qquad\implies\qquad \Leb(W_{n,m}^P(\psi)) =1.
    \end{equation*}
    
  \item If $P$ is relatively uniformly discrete, then for bounded
    functions $\psi:\Zpos^n\to\Rpos$,
    \begin{equation*}
      \inf_{Q}\Leb(W_{1,m}^P(\Psi_Q)) > 0 \qquad\implies\qquad \Leb(W_{n,m}^P(\psi)) =1.
    \end{equation*}
  \end{enumerate}
\end{theorem}

The theorem is vacuously true when $n=1$. This is because for all
integers $q \geq 1$, one has $\abs{q/\gcd(q)} = 1$, and so if $n=1$
then $\Psi_Q$ is defined by an empty sum as soon as
$Q>1$. Consequently, if $Q>1$, then $W_{1,m}^P(\Psi_Q)$ has zero
measure. To put it another way, \th\ref{thm:bootstrap11} is only
meaningful if one increases dimension from $n=1$ to $n>1$. In fact,
the condition on $\Psi_Q$ ensures that the set $W_{n,m}^P(\psi)$ is
genuinely higher-dimensional. Without it, one can reproduce a
$1$-by-$m$-dimensional set in $\TT^{nm}$ by choosing $\psi$ to be
supported on a line in $\q\in\Zpos^n$ (see \th\ref{ex:postopos}).

The applications in~\S\ref{sec:applications} all begin with some
$1$-by-$m$-dimensional knowledge---a \emph{simultaneous theorem}. For
the applications that require \th\ref{thm:bootstrap11}, the
simultaneous theorem already takes care of the cases where the
condition on $\Psi_Q$ does not hold. In other words, the
$1$-by-$m$-dimensional artifice described in the previous paragraph
represents a problem that has already been solved in the $1$-by-$m$
setting.

\subsection{About the proofs}
\label{sec:aboutproofs}

\paragraph{Positive measure.}

The problem of showing that a limsup set $W= \limsup_{q\to\infty}A_q$
has positive measure in a probability space $(X,\mu)$ always revolves
around establishing some form of stochastic independence among the
sets $A_q$. For the sets that one typically encounters in metric
number theory, the best one can hope for is a weak form of
independence called quasi-independence on average. It is defined by
\begin{equation*}
  \limsup_{D\to\infty}\frac{\parens*{\sum_{k \leq D} \mu(A_k)}^2}{\sum_{k ,\ell \leq D} \mu(A_k\cap A_\ell)}>0
\end{equation*}
and it is enough to guarantee positive measure of
$W$. (\th~\ref{prop:erdoschung,prop:erdosrenyi,prop:qia} all express
this idea in various forms.) In fact, a partial converse is true:
Beresnevich and Velani~\cite{BVBC} have recently shown that if a
sequence of balls has a positive-measure limsup set, then there is a
subsequence of those balls exhibiting quasi-independence on
average. (See \th~\ref{thm:BV}.) One of the main steps in the proofs
of \th\ref{thm:bootstrap,thm:bootstrap11} hinges on an application of
this result.

The sets of interest in \th~\ref{thm:bootstrap,thm:bootstrap11} are 
\begin{equation*}
W_{n,m}^P(\psi) = \limsup_{\abs{\q}\to\infty} A_{n,m}^P(\q, \psi(\q)),
\end{equation*}
where $A_{n,m}^P(\q,\psi(\q))$ denotes the set of points in $\TT^{nm}$
for which~\eqref{eq:one} holds. We work under the assumption that
$\Leb(W_{1,m}^P(\Psi))>0$, and the goal is to show that the sets
$A_{n,m}^P(\q,\psi(\q))\,(\q\in\Zpos^n)$ are quasi-independent on
average. We note that the sets $A_{1,m}^P(d,\Psi(d))$ are unions of
balls, so $W_{1,m}^P(\Psi)$ can be viewed as the limsup set of a
sequence of balls. Thus, the above mentioned Beresnevich--Velani
result furnishes a quasi-independent subsequence of those balls, which
in turn corresponds to a refinement $R\subset P$ such that the sets
$A_{1,m}^R(d,\Psi(d))\, (d\geq 1)$ are quasi-independent on
average. (See \th~\ref{lem:BVtranslation}.) By modifying a strategy
in~\cite{ramirez2023duffinschaeffer}, we are able to leverage the
quasi-independence of the sets $A_{1,m}^R(d,\Psi(d))$ in the
calculation of average pairwise overlaps
\begin{equation*}
  \Leb\parens*{A_{n,m}^R(\q,\psi(\q))\cap A_{n,m}^R(\br,\psi(\br))}
\end{equation*}
in order to achieve quasi-independence on average of
$A_{n,m}^R(\q,\psi(\q))\,(\q\in\Zpos^n)$. This proves positive measure
as needed in \th~\ref{thm:bootstrap}.

\paragraph{Full measure.}

Once positive measure of $W_{n,m}^P(\psi)$ is secured, the idea is to
fix an arbitrary small open set $U\subset \TT^{nm}$, and show that
\begin{equation}\label{eq:58}
  \Leb\parens*{W_{n,m}^P(\psi)\cap U} \geq c \Leb (U)^2,
\end{equation}
where $c>0$ is a constant that does not depend on $U$. Then a variant
of the Lebesgue density theorem (Beresnevich--Dickinson--Velani,
\th~\ref{lem:lebesguedensity}) shows that $W_{n,m}^P(\psi)$ must have
full measure.

In order to accomplish that, we examine the sets
$A_{n,m}^P(\q,\psi(\q))$ on small scales. In \th~\ref{lem:regularity}
we show that if $\abs{\q/\gcd(\q)}$ is sufficiently large---say,
larger than some $Q\geq 1$ depending on $U$---then
\begin{equation}\label{eq:59}
  \Leb\parens*{A_{n,m}^P(\q,\psi(\q))\cap U} \geq \frac{1}{3}\Leb(A_{n,m}^P(\q,\psi(\q)))\Leb(U).
\end{equation}
We then run the argument from the proof of \th\ref{thm:bootstrap}, but
with $\Psi_Q$ instead of $\Psi_1$. Propagating~\eqref{eq:59} through
the resulting independence calculations for
$A_{n,m}^R(\q,\psi(\q))\,(\q\in\Zpos^n)$ gives~\eqref{eq:58} with a
constant $c>0$ that depends on $Q$, and hence on $U$. In fact, it
depends on $\Leb(W_{n,m}^P(\Psi_Q))$ in an explicit way, as can be seen
by keeping track of the implicit constant in the quasi-independence
from the Beresnevich--Velani result. Now the assumption that
$\inf_Q \Leb(W_{n,m}^P(\Psi_Q))>0$ allows the whole argument to work
with a constant that is uniform over $Q$, and this leads to the
desired $c>0$---independent of $U$---for which~\eqref{eq:58} holds,
proving \th~\ref{thm:bootstrap11}.

\section{Applications}
\label{sec:applications}

The general outline for the applications of
\th\ref{thm:bootstrap,thm:bootstrap11} is this: Given a function
$\psi$ satisfying a divergence condition in the $n$-by-$m$-setting,
define the associated functions $\Psi_Q$. If there exists $d, Q\geq 1$
for which $\Psi_Q(d)=\infty$, then the following proposition gives the
desired result.

\begin{proposition}\th\label{prop:Psifinite}
  Fix $m,n \in\NN$. For each $d\geq 1$, let
  $P(d) \subset (\RR/d\ZZ)^m$ be a finite set, and for each
  $\q\in\ZZ^n$ with $\gcd(\q)=d$, let $P(\q)$ be the lift to $\RR^m$
  of $P(d)$. Let $\psi: \Zpos^n \to \Rpos$. If there exist $d,Q\geq 1$
  for which $\Psi_Q(d) = \infty$, then $\Leb(W_{n,m}^P(\psi))=1$.
\end{proposition}

\begin{remark*}
  The proof is in \S\ref{sec:reductions}. The result can also be
  deduced from~\cite[Chapter~1, Theorem~14]{Sprindzuk}.
\end{remark*}
  
If, on the other hand, $\Psi_Q(d)$ is always finite, then we show that
$\Psi_1$ (or $\Psi_Q$) satisfies the divergence condition from a
\emph{simultaneous theorem} in the $1$-by-$m$-setting. The
simultaneous theorem then gives $\Leb(W_{1,m}^P(\Psi_1))>0$ (or
$\inf_Q\Leb(W_{1,m}^P(\Psi_Q))>0$), and \th\ref{thm:bootstrap} (or
\th\ref{thm:bootstrap11}) bootstraps this to $\Leb(W_{n,m}^P(\psi))>0$
(or $\Leb(W_{n,m}^P(\psi))=1$).

Some of these applications are short proofs of existing results, some
are new results, and some are hypothetical in the sense that they rely
on hypothetical $1$-by-$m$-dimensional theorems. The new theorems are
numbered.

\subsection{Khintchine 1924, 1926 $\implies$ Khintchine--Groshev 1938}
\label{sec:khintch-1924-impl}

Khintchine's theorem was proved in 1924 in dimension $1$ and in 1926
for higher dimensions~\cite{Khintchine,Khintchineonedimensional}. In
1938, Groshev proved what is now known as the Khintchine--Groshev
theorem~\cite{Groshev}---the analogue of Khintchine's theorem for
$n$-by-$m$ systems of linear forms. 

\begin{theorem*}[Khintchine--Groshev theorem]
  Fix $m,n\in\NN$. With $P(\q) = \ZZ^m$ for every $\q\in\ZZ^n$, and
  $\psi(\q):=\psi(\abs{\q})$ depending only on $\abs{\q}$,
  \begin{equation*}
    \Leb(W_{n,m}^P(\psi)) =
    \begin{cases}
      0 &\textrm{if } \sum_{q=1}^\infty q^{n-1}\psi(q)^m  < \infty \\
      1 &\textrm{if } \sum_{q=1}^\infty q^{n-1}\psi(q)^m  = \infty \textrm{ and } \psi(q) \textrm{ is nonincreasing.}
    \end{cases}
  \end{equation*}
\end{theorem*}

One can use \th\ref{thm:bootstrap} and a zero-one law due to
Beresnevich and Velani~\cite[Theorem~1]{BVzeroone} to show that
Khintchine's theorem implies the Khintchine--Groshev theorem. In fact,
it implies the following stronger statement.

\begin{theorem}\th\label{thm:khintPsi}
  Fix $m,n\in\NN$. Let $P(\q) = \ZZ^m$ for every $\q\in\ZZ^n$, and
  $\psi:\Zpos^n\to\Rpos$.  Then
  \begin{equation*}
    \Leb(W_{n,m}^P(\psi)) =
    \begin{cases}
      0 &\textrm{if } \sum_{\q\in\Zpos^n} \psi(\q)^m  < \infty \\
      1 &\textrm{if } \sum_{\q\in\Zpos^n} \psi(\q)^m = \infty \textrm{ and } \Psi \textrm{ is nonincreasing,}
    \end{cases}
  \end{equation*}
  where $\Psi:=\Psi_1$ is defined by~\eqref{eq:Psidef}.
\end{theorem}

\begin{proof}
  For a fixed $\q\in \Zpos^n$, the set of $\X\in \TT^{nm}$ for
  which~\eqref{eq:one} holds has measure bounded from above by
  $(2\psi(\q))^m$. Therefore, the Borel--Cantelli lemma immediately
  proves the convergence part of this theorem.

  Suppose then that $\sum_{\q\in\Zpos^n}\psi(\q)^m$ diverges. If there
  is some $d\geq 1$ such that $\Psi(d)=\infty$, then the result
  follows from \th~\ref{prop:Psifinite}. Otherwise, note that
  \begin{equation*}
    \sum_{d\geq 1} \Psi(d)^m =  \sum_{\q\in\Zpos^n}\psi(\q)^m = \infty. 
  \end{equation*}
  Since it is assumed that $\Psi(d)$ is nonincreasing, Khintchine's
  theorem gives $\Leb(W_{1,m}^P(\Psi)) = 1$. Clearly,
  $P$ is well-spread in the sense of \th~\ref{def:well}, therefore
  \th\ref{thm:bootstrap} applies. It gives that
  $\Leb(W_{n,m}^P(\psi))>0$. Full measure follows from the zero-one
  law~\cite[Theorem~1]{BVzeroone}.
\end{proof}

\begin{proof}[Proof of the Khintchine--Groshev theorem]
  Let $\psi:\NN\to\Rpos$, and define $\bar\psi:\Zpos^n\to\Rpos$ by
  $\bar\psi(\q)=\psi(\abs{\q})$. We will show that
  \th~\ref{thm:khintPsi} applies to the function $\bar\psi$.
    
  Since
  \begin{equation}\label{eq:56}
    \sum_{q=1}^\infty q^{n-1}\psi(q)^m \asymp \sum_{q=1}^\infty\sum_{\abs{\q}=q} \psi(q)^m  \asymp \sum_{\q\in\Zpos^n}\bar\psi(\q)^m,
  \end{equation}
  the convergence case follows from the convergence case of
  \th~\ref{thm:khintPsi}.

  For the divergence case, notice that~\eqref{eq:56} implies
  $\sum_{\q\in\Zpos^n}\bar\psi(\q)^m$ diverges. Therefore, in order to
  apply \th~\ref{thm:khintPsi}, we only need to verify that $\Psi$ is
  nonincreasing. To see this, observe that for each $d\geq 1$,
  \begin{equation}\label{eq:57}
    \Psi(d) = \parens*{\sum_{\gcd(\q)=1}\bar\psi(d\q)^m}^{1/m}.
  \end{equation}
  Recall that in the divergence case it is assumed that $\psi(q)$ is
  nonincreasing. Therefore, for every $\q$, one has that
  $\bar\psi(d\q):= \psi(d\abs{\q})$ is a nonincreasing function of
  $d\geq 1$. Reviewing~\eqref{eq:57}, one sees that therefore $\Psi$
  is nonincreasing.
\end{proof}

The next application shows how \th\ref{thm:bootstrap11} can be used in
cases where there is no pre-existing zero-one law.

\subsection{Sz{\"u}sz 1958 and Schmidt 1964 $\implies$ Sprindzuk 1979}
\label{sec:inhom-khintch-impl}

The inhomogeneous version of Khintchine's theorem was proved in 1958
by Sz{\"u}sz in dimension $m=1$, and in 1964 by Schmidt in higher
dimensions~\cite{SzuszinhomKT,Schmidt}. Combining these with
\th\ref{thm:bootstrap11} gives the following inhomogeneous version of
Khintchine--Groshev, which can be found in
Sprindzuk~\cite[1979]{Sprindzuk}.

\begin{theorem*}[Inhomogeneous Khintchine--Groshev theorem]
  Fix $m,n\in\NN$. With $P(\q) = \y + \ZZ^m$ for every $\q\in\ZZ^n$,
  where $\y\in\RR^m$ a constant, and $\psi(\q):=\psi(\abs{\q})$
  depending only on $\abs{\q}$,
  \begin{equation*}
    \Leb(W_{n,m}^P(\psi)) =
    \begin{cases}
      0 &\textrm{if } \sum_{q=1}^\infty q^{n-1}\psi(q)^m  < \infty \\
      1 &\textrm{if } \sum_{q=1}^\infty q^{n-1}\psi(q)^m  = \infty \textrm{ and } q^{n-1}\psi(q)^m \textrm{ is nonincreasing.}\\
    \end{cases}
  \end{equation*}
\end{theorem*}

\begin{proof}[Proof modulo the $n = 1$ case]
  Only the divergence case requires attention. The convergence case is
  identical to that of \th~\ref{thm:khintPsi}: it is an application of
  the Borel--Cantelli lemma.

  Let $\psi:\NN \to \Rpos$ be a function such that $q^{n-1}\psi(q)^m$ is
  nonincreasing and such that
  \begin{equation}\label{eq:1}
    \sum_{q=1}^\infty q^{n-1}\psi(q)^m = \infty.
  \end{equation}
  Notice that $\psi$ is a nonincreasing function. Extend $\psi$ to
  $\psi:\Zpos^n\to\Rpos$ such that $\psi(\q):=\psi(\abs{\q})$.

  Note that $P(\q) = \y + \ZZ^m$ is the lift to $\RR^m$ of
  \begin{equation*}
    P(\gcd(\q)) = \y + \parens*{\ZZ/\gcd(\q)\ZZ}^m \subset \parens*{\RR/\gcd(\q)\ZZ}^m 
  \end{equation*}
  for all $\q\in\ZZ^n$. It is easily seen that $P$ is well-spread in
  the sense of Definition~\ref{def:well}.

  For $d,Q\geq 1$ define $\Psi_Q(d)$ by~(\ref{eq:Psidef}). If there
  exists $d,Q\geq 1$ for which $\Psi_Q(d)=\infty$, then
  \th\ref{prop:Psifinite} gives that $\Leb(W_{n,m}^P(\psi))=1$.
  Otherwise, we have as a consequence of~(\ref{eq:1}) that the series
  $\sum_d \Psi(d)^m$ diverges. Notice now that
  \begin{equation*}
    \sum_{d\geq 1}\Psi_Q(d)^m = \sum_{\substack{\q\in\Zpos^n 
    \\ \abs{\q/\gcd(\q)}\geq Q}} \psi(\q)^m \\
                            = \sum_{q=1}^\infty \sum_{\substack{\abs{\q}=q 
    \\ \abs{\q/\gcd(\q)}\geq Q}} \psi(\q)^m.
  \end{equation*}
  Let $\PP_Q^n = \set{\q\in\Zpos^n : \gcd(\q)=1, \abs{\q} < Q}$. It is
  a finite set. Let $M$ be large enough that
  $\#\set{\q\in\Zpos^n : \abs{\q}=q}\geq 2\#\PP_Q^n$ for all
  $q\geq M$. Then, continuing the above calculations,
  \begin{align*}
    \sum_{d\geq 1}\Psi_Q(d)^m &\geq \sum_{q=M}^\infty \sum_{\substack{\abs{\q}=q \\
    \abs{\q/\gcd(\q)}\geq Q}} \psi(\q)^m \\
    &\geq \sum_{q=M}^\infty \psi(q)^m \parens*{\#\set{\q\in\Zpos^n :
    \abs{\q}=q} - \#\PP_Q^n} \\
    &\geq \frac{1}{2}\sum_{q=M}^\infty \psi(q)^m \#\set{\q\in\Zpos^n :
    \abs{\q}=q} \\
    &\asymp \sum_{q=M}^\infty q^{n-1}\psi(q)^m.
  \end{align*}
  Now, by~(\ref{eq:1}) we have that $\sum_d \Psi_Q(d)^m$ diverges. The
  monotonicity of $\psi$ implies that
  \begin{equation*}
    \Psi_Q(d) = \parens*{\sum_{\q\in\PP_\infty^n\setminus\PP_Q^n} \psi(d\q)^m}^{1/m}
  \end{equation*}
  is nonincreasing, since for every fixed
  $\q\in\PP_\infty^n\setminus\PP_Q^n$, the function
  $\psi(d\q) = \psi(d\abs{\q})$ is nonincreasing.  The inhomogeneous
  version of Khintchine's theorem now gives $\Leb(W_{1,m}(\Psi_Q))=1$,
  and by \th\ref{thm:bootstrap11}, $\Leb(W_{n,m}^P(\psi))=1$.
\end{proof}

\subsection{Gallagher 1965 $\implies$ Beresnevich--Velani 2010}
\label{sec:gall-1965-impl}

Famously, Khintchine's theorem requires a monotonicity assumption in
dimension $1$. But in 1965 Gallagher~\cite{Gallagherkt} showed that
that assumption is unnecessary in higher dimensions. Sprindzuk removed
the monotonicity assumption from the Khintchine--Groshev theorem in
the cases where $n>2$, and in 2010 Beresnevich--Velani were able to
remove the monotonicity assumption from the Khintchine--Groshev
theorem in all cases where $mn>1$~\cite{Sprindzuk,BVKG}. They also
proved a multivariate analog of the Khintchine--Groshev theorem
without monotonicity assumptions, for $m>1$. \th\ref{thm:bootstrap11}
can be used to derive this result from its $n=1$ cases, that is, from
Gallagher's 1965 result. In fact, \th\ref{thm:bootstrap} is sufficient
for this purpose, thanks again to the zero-one
law~\cite[Theorem~1]{BVzeroone}.

\begin{theorem*}[Beresnevich--Velani,~{\cite[Theorem~5]{BVKG}}]
  Fix $m,n\in\NN$ with $m>1$. Let $P(\q) = \ZZ^m$ for every
  $\q\in\ZZ^n$, and $\psi:\ZZ^n \to \Rpos$. Then
  \begin{equation*}
    \Leb(W_{n,m}^P(\psi)) =
    \begin{cases}
      0 &\textrm{if } \sum_{\q\in\ZZ^n} \psi(\q)^m  < \infty \\
      1 &\textrm{if } \sum_{\q\in\ZZ^n} \psi(\q)^m = \infty.
    \end{cases}
  \end{equation*}
\end{theorem*}

\begin{proof}[Proof modulo the $n=1$ case]
  The convergence case of this theorem is an application of the
  Borel--Cantelli lemma, so we only need to discuss the divergence
  case.

  Assume that $\sum_{\q\in \ZZ^n}\psi(\q)^m$ diverges. There is some
  orthant of $\ZZ^n$ such that the series still diverges when
  restricted to that orthant, and by applying coordinate reflections
  we may assume without loss of generality that it is the positive
  orthant, $\Zpos^n$. So let us replace $\psi$ with
  $\bone_{\Zpos^n}\psi$.

  Define $\Psi = \Psi_1$ by~(\ref{eq:Psidef}). If there exists
  $d\geq 1$ for which $\Psi(d)=\infty$, then the theorem follows from
  \th\ref{prop:Psifinite}. Otherwise, observe that
  \begin{equation*}
    \sum_{d\geq 1}\Psi(d)^m  = \sum_{\q\in\Zpos^n}\psi(\q)^m
  \end{equation*}
  diverges. Since $m > 1$, Gallagher's extension of Khintchine's
  theorem gives $\Leb(W_{1,m}^P(\Psi)) = 1$. By
  \th\ref{thm:bootstrap}, $\Leb(W_{n,m}^P(\psi))>0$. Applying once
  again Beresnevich and Velani's zero-one law for systems of linear
  forms~\cite[Theorem~1]{BVzeroone}, we have
  $\Leb(W_{n,m}^P(\psi)) = 1$.
\end{proof}

As an immediate corollary, one gets the $m>1$ cases
of~\cite[Theorem~1]{BVKG}, the version where $\psi(\q)$ is constant on
spheres $\abs{\q} = q$.

\subsection{Yu 2019 $\implies$ multivariate version of
  Allen--Ram{\'i}rez 2023 ($m>2$)}
\label{sec:yu-2019-implies}

In 2019, Yu proved that the inhomogeneous Khintchine theorem does not
need a monotonicity assumption in dimension
$m>2$~\cite[Theorem~1.8]{Yu}. In 2023, Allen and the author showed
that one can remove the monotonicity from the inhomogeneous
Khintchine--Groshev theorem whenever
$nm>2$~\cite[Theorem~1]{allen2021independence}. The $m>2$ cases of
this result can now be proved by applying \th\ref{thm:bootstrap11} to
Yu's result. In fact, one can prove the following multivariate
theorem, an inhomogeneous version of
Beresnevich--Velani~\cite[Theorem~5]{BVKG} for $m>2$.

\begin{theorem}\th\label{thm:multivariateAR}
  Fix $m,n\in\NN$ with $m>2$. Let $\y\in\RR^m$ and
  $P(\q) = \y + \ZZ^m$ for every $\q\in\ZZ^n$, and
    $\psi:\ZZ^n \to \Rpos$. Then
  \begin{equation*}
    \Leb(W_{n,m}^P(\psi)) =
    \begin{cases}
      0 &\textrm{if } \sum_{\q\in\ZZ^n} \psi(\q)^m  < \infty \\
      1 &\textrm{if } \sum_{\q\in\ZZ^n} \psi(\q)^m = \infty.
    \end{cases}
  \end{equation*}
\end{theorem}

\begin{proof}
  The convergence part is a standard application of the
  Borel--Cantelli lemma, so we will concentrate on the divergence
  part.

  As in the proof from \S\ref{sec:gall-1965-impl}, no generality is
  lost in treating the case where $\psi:\Zpos^n\to\Rpos$ such that
  $\sum_{\q\in\Zpos^n}\psi(\q)^m$ diverges.

  Let $\Psi(d)$ be defined as in \th\ref{thm:bootstrap}. If there
  exists $d\geq 1$ for which $\Psi(d)=\infty$, then we are done by
  \th\ref{prop:Psifinite}. Assume then that $\Psi(d)$ is always
  finite.

  If there exists $\q\in\Zpos^n$ with $\gcd(\q)=1$ such that
  \begin{equation}\label{eq:25i}
    \sum_{d\geq 1} \psi(d\q)^m = \infty,
  \end{equation}
  then by Yu's result, $\Leb(W_{1,m}^P(\psi_\q))=1$, where
  $\psi_\q:\Zpos\to\Rpos$ is defined by $\psi_\q(d) = \psi(d\q)$. But
  in this case
  \begin{equation*}
     T^{-1}\parens*{W_{1,m}^P(\psi_\q)}\subset W_{n,m}^P(\psi) 
  \end{equation*}
  where $T:\TT^{mn}\to \TT^m$ is the map $\X \mapsto \q\X \pmod
  1$. That map preserves measure, therefore,
  $\Leb(W_{n,m}^P(\psi))=1$.

  On the other hand, suppose there is no $\q\in\Zpos^n$ with
  $\gcd(\q)=1$ for which~(\ref{eq:25i}) holds. Then for every
  $Q\geq 1$,
  \begin{equation*}
    \sum_{d\geq 1} \sum_{\substack{\gcd(\q)=1 \\ \abs{\q} < Q}}
    \psi(d\q)^m  < \infty
  \end{equation*}
  therefore
  \begin{equation*}
    \sum_{d\geq 1}\Psi_Q(d)^m
    = \sum_{d\geq
      1}\Psi(d)^m
      -\sum_{d\geq 1}\sum_{\substack{\gcd(\q)=1 \\ \abs{\q} <
    Q}}\psi(d\q)^m = \infty
  \end{equation*}
  where $\Psi_Q$ is defined by~(\ref{eq:Psidef}). Now,
  by~\cite[Theorem~1.8]{Yu}, $\Leb(W_{1,m}^P(\Psi_Q)) = 1$. Since this
  holds for every $Q\geq 1$, \th\ref{thm:bootstrap11} implies that
  $\Leb\parens*{W_{n,m}^P(\psi)} = 1$, finishing the proof.
\end{proof}

As a corollary, one gets the univariate version of
\th~\ref{thm:multivariateAR}, where $\psi$ depends only on
$\abs{\q}$---that is, the inhomogeneous Khintchine--Groshev theorem
without monotonicity in all cases where $m>2$. This was proved for all
$nm>2$ in~\cite[Theorem~1]{allen2021independence}. In fact, one can
prove a version where the inhomogeneous parameter $\y\in\RR^m$ can
vary with $\q\in\ZZ^n$.

In~\cite[Conjecture~1]{allen2021independence} it is conjectured that
it should be possible to remove the monotonicity condition from the
inhomogeneous univariate Khintchine--Groshev theorem in the remaining
$1$-by-$2$ and $2$-by-$1$ cases. The $1$-by-$2$ case of the
conjecture, combined with \th\ref{thm:bootstrap11}, would imply that
\th~\ref{thm:multivariateAR} also holds for $m=2$, bringing it in line
with the homogeneous theorem of Beresnevich--Velani
in~\S\ref{sec:gall-1965-impl}.

\subsection{Harman 1988 $\implies$ Nesharim--R{\"u}hr--Shi 2020}
\label{sec:nesharim-ruhr-shi}

In 1988 Harman considered approximation of real numbers by rational
numbers $p/q$ whose numerators and denominators lie in prescribed
arithmetic sequences, say $p\equiv r\bmod a$ and $q\equiv s \bmod b$
where $a,b \geq 1$, $0 \leq r < a-1$, and $0\leq s < b-1$ are fixed
integers. He found asymptotics for the number of such rational
approximates to typical real numbers~\cite{HarmanRestricted}, which
yield in particular a Khintchine-type metric dichotomy like the ones
discussed here. That dichotomy was extended to the linear forms
setting in 2020 by Nesharim, R{\"u}hr, and Shi~\cite{NesharimRuhrShi}.

Fix $m,n\in\NN$. For a function $\psi:\ZZ^n \to\Rpos$ and fixed
vectors $(\ba,\bb) \in \NN^m\times \NN^n$ and
$(\br,\bs) \in\ZZ^m\times\ZZ^n$, define
\begin{equation*}
  \calK_{n,m}(\psi, \ba, \bb, \br, \bs) = \set*{\X \in M_{n\times m}(\RR) :\substack{\textrm{ there are infinitely many }
      (\p, \q)\in \ZZ^m\times\ZZ^n \\ \textrm{ satisfying }
      \abs{\q\X - \p} < \psi(\q)
      \textrm{ such that} \\ \p \equiv \br\, (\bmod\,  \ba) \textrm{ and } \q \equiv \bs\, (\bmod\, \bb)}}
\end{equation*}
where $\p \equiv \br\, (\bmod\,  \ba)$ means entrywise
congruence. Nesharim, R{\"u}hr, and Shi prove the following
theorem. Its $n>1$ cases can be obtained from its $n=1$ case using
\th\ref{thm:bootstrap11}.

\begin{theorem*}[Nesharim--R{\"u}hr--Shi~{\cite[Theorem~1.2]{NesharimRuhrShi}}]
  Fix $m,n\in\NN$. Let $\psi:\NN \to\Rpos$ be nonincreasing, and
  fix $(\ba,\bb) \in \NN^m\times \NN^n$ and
  $(\br,\bs) \in\ZZ^m\times\ZZ^n$. Then
  \begin{equation*}
    \Leb\parens*{\calK_{n,m}(\psi, \ba, \bb, \br, \bs)} =
    \begin{cases}
      \textsc{zero} &\textrm{if } \sum_{q=1}^\infty q^{n-1}\psi(q)^m < \infty  \\
      \textsc{full} &\textrm{if } \sum_{q=1}^\infty q^{n-1}\psi(q)^m = \infty
    \end{cases}
  \end{equation*}
  where $\psi(\q):=\psi(\abs{\q})$.
\end{theorem*}

\begin{remark*}
  The reason this theorem is given as a \textsc{zero-full} statement
  instead of a $0$-$1$ statement is that the set $\calK_{n,m}$ is not
  $\ZZ^{mn}$-periodic in general. It is, however,
  $(a\ZZ)^{mn}$-periodic where $a=\lcm(\ba)$. In the following proof,
  we scale the set by $1/a$, resulting in a $\ZZ^{mn}$-periodic setup
  in which we apply~\th~\ref{thm:bootstrap11}.
\end{remark*}

\begin{proof}[Proof modulo the $n=1$ case]
  The convergence case follows from the convergence case of the
  Khintchine--Groshev theorem, so only the divergence case needs
  proving.

  Let
  $\ba=(a_1, \dots, a_m),\bb = (b_1, \dots, b_n), \br = (r_1, \dots,
  r_m)$, and $\bs=(s_1, \dots, s_n)$ be as in the theorem
  statement. Assume, without loss of generality, that
  $0\leq r_i \leq a_i\, (i=1, \dots, m)$ and
  $0\leq s_j \leq b_j, (j = 1, \dots, n)$. Let $a = \lcm(\ba)$ and
  $b = \lcm(\bb)$. Then for all $p\in\ZZ$ and $i=1,2,\dots, m$ we have
  \begin{equation*}
    p \equiv r_i\, (\bmod\, a) \implies p \equiv r_i \, (\bmod\, a_i),
  \end{equation*}
  hence, for all $\p \in \ZZ^m$ we have
  \begin{equation*}
    \p \equiv \br\, (\bmod\, \hat\ba) \implies \p \equiv \br \, (\bmod\, \ba),
  \end{equation*}
  where $\hat\ba = (a, a, \dots, a)$.  Therefore, it suffices to show
  $\Leb\parens*{\calK_{n,m}(\psi, \hat\ba, \bb, \br, \bs)} =
  \textsc{full}$.

  For each $\q\in\Zpos^n$, put
  \begin{equation*}
    P(\q) = \frac{\br}{a} + \ZZ^m
  \end{equation*}
  and note that this defines a well-spread $P$, as in
  Definition~\ref{def:well}. Define $\bar\psi:\Zpos^n\to\Rpos$ by
  \begin{equation*}
    \bar\psi(\q) =
    \begin{cases}
      \frac{1}{a}\psi(\abs{\q}) &\textrm{if } \q \equiv \bs \, (\bmod \bb)\\
      0 &\textrm{otherwise}
    \end{cases}
  \end{equation*}
  and for each $0\leq s \leq b-1$, define
    \begin{equation*}
    \bar\psi_s(\q) =
    \begin{cases}
      \bar\psi(\q) &\textrm{if } \gcd(\q) \equiv s \, (\bmod b)\\
      0 &\textrm{otherwise,}
    \end{cases}
  \end{equation*}
  so that $\bar\psi = \bar\psi_0 + \dots + \bar\psi_{b-1}$. Since $\psi$ is
  nonincreasing and $\sum q^{n-1}\psi(q)^m$ diverges, it is
  readily verified that
  \begin{equation*}
    \sum_{\q\in\Zpos^n}\bar\psi(\q)^m
  \end{equation*}
  also diverges.

  Suppose there is some $s \in \set{0, \dots, b-1}$ and
  $\q' \in \Zpos^n$ with $\gcd(\q')=1$ such that
  \begin{equation}\label{eq:53}
    \sum_{d\geq 1}\bar\psi_s(d\q')^m = \infty. 
  \end{equation}
  Then 
  \begin{equation*}
    \sum_{d\geq 1}\psi(d\abs{\q'})^m = \infty,
  \end{equation*}
  and by the $n=1$ case,
  $\Leb(\calK_{1,m}(\psi_{\q'}, \hat\ba, b, \br, s)) = \textsc{full}$,
  where $\psi_{\q'}(d): = \psi(d\abs{\q'})$. But note that the
  scaled set $\frac{1}{a}\calK_{1,m}(\psi_{\q'}, \hat\ba, b, \br, s)$
  is the lift to $\RR^m$ of $W_{1,m}^P(\bar\psi_{s,\q'})$, where $\bar\psi_{s,\q'}(d):= \bar\psi_s (d\abs{\q'})$, and so
  $\Leb\parens{W_{1,m}^P(\bar\psi_{s,\q'})} = 1$. But
  \begin{equation*}
    T_{\q'}^{-1}\parens{W_{1,m}^P(\bar\psi_{s,\q'})} \subset W_{n,m}^P(\bar\psi_s), 
  \end{equation*}
  where $T_{\q'}:\TT^{mn}\to\TT^m$ is the projection
  $\X\mapsto \q' \X$. Therefore,
  $\Leb\parens{W_{n,m}^P(\bar\psi_s)} = 1$, since $T_{\q'}$ is
  measure-preserving. And the lift of $W_{n,m}^P(\bar\psi_s)$ to
  $M_{n\times m}(\RR)$ is contained in the scaled set
  $\frac{1}{a}\calK_{n,m}(\psi, \hat\ba, \bb, \br, \bs)$, so
  $\Leb\parens{\calK_{n,m}(\psi, \hat\ba, \bb, \br,
    \bs)}=\textsc{full}$.

  Suppose, on the other hand, that there are no
  $s\in\set{0,\dots, b-1}$ and $\gcd(\q')=1$ such that~(\ref{eq:53})
  holds. Choose $s$ so that
  \begin{equation*}
    \sum_{\q\in\Zpos^n}\bar\psi_s(\q)^m = \infty,
  \end{equation*}
  and let $Q\geq 1$. Define $\Psi_{s,Q}$ by~(\ref{eq:Psidef}) applied
  to $\bar\psi_s$. Since~(\ref{eq:53}) holds for no $\q'\in\Zpos^n$,
  it follows that
  \begin{equation*}
    \sum_{d\geq 1}\Psi_{s,Q}(d)^m = \infty.
  \end{equation*}
  Note that $\Psi_{s,Q}$ is supported on $d\equiv s\, (\bmod\,
  b)$. Let $\widehat\Psi_{s,Q}(d)$ be a nonincreasing function such
  that $\widehat\Psi_{s,Q}(d)=\Psi_{s,Q}(d)$ for all
  $d\equiv s\, (\bmod\, b)$. Then
  \begin{equation*}
    \sum_{d\geq 1}\widehat\Psi_{s,Q}(d)^m = \infty,
  \end{equation*}
  so the $n=1$ case of this theorem implies that
  $\Leb\parens{\calK_{1,m}(a\widehat\Psi_{s,Q}, \hat\ba, b, \br, s)} =
  \textsc{full}$. But
  \begin{equation*}
    \calK_{1,m}(a\widehat\Psi_{s,Q}, \hat\ba, b, \br, s) =
    \calK_{1,m}(a\Psi_{s,Q}, \hat\ba, b, \br, s),
  \end{equation*}
  so we have
  $\Leb\parens{\calK_{1,m}(a\Psi_{s,Q}, \hat\ba, b, \br, s)} =
  \textsc{full}$.  Now, the scaled set
  $\frac{1}{a}\calK_{1,m}(a\Psi_{s,Q}, \hat\ba, b, \br, s)$ is the
  lift to $\RR^m$ of $W_{1,m}^P(\Psi_{s,Q})$, so
  $\Leb\parens{W_{1,m}^P(\Psi_{s,Q})}=1$. Since $Q\geq 1$ was
  arbitrary, \th\ref{thm:bootstrap11} implies that
  $\Leb\parens{W_{n,m}^P(\bar\psi_s)}=1$, so
  $\Leb\parens{W_{n,m}^P(\bar\psi)}=1$. Finally, the lift of
  $W_{n,m}^P(\bar\psi)$ to $M_{n\times m}(\RR)$ is a scaling of
  $\calK_{n,m}(\psi, \hat\ba, \bb, \br, \bs)$ by $1/a$, so it follows
  that
  $\Leb\parens{\calK_{n,m}(\psi, \hat\ba, \bb, \br,
    \bs)}=\textsc{full}$ and the proof is finished.
\end{proof}

Adiceam proves the $m=n=1$ case of this theorem with the additional
condition $\gcd(p,q) = \gcd(a,b,r,s)$ imposed in the definition of
$\calK_{1,1}(\psi, a, b, r, s)$ (see
\cite[Theorem~2]{adiceamIJNT}). \th\ref{thm:bootstrap11} can be used
to prove the dual version of that statement, that is, the $m=1$ cases
of the above theorem, with the additional condition
$\gcd(p,\q) = \gcd(a,\bb,r,\bs)$.

\subsection{Inhomogeneous Duffin--Schaeffer $\implies$ same for
  systems of linear forms}
\label{sec:inhom-duff-scha}

In 1941, Duffin and Schaeffer showed by counterexample that the
monotonicity condition cannot be removed from the one-dimensional
version of Khintchine's theorem. In the same paper, they formulated
what became known as the Duffin--Schaeffer
conjecture~\cite{duffinschaeffer}, a problem that stood eight
decades. It was finally proved in a 2020 breakthrough by
Koukoulopoulos and Maynard~\cite{KMDS}. A higher-dimensional version
had been proved in 1990 by Pollington and
Vaughan~\cite{PollingtonVaughan}. The following theorem is the version
of the Duffin--Schaeffer conjecture for systems of linear forms. It
was conjectured by Beresnevich--Bernik--Dodson--Velani~\cite{BBDV} in
2009 and recently proved by the
author~\cite{ramirez2023duffinschaeffer}.

\begin{theorem*}[Duffin--Schaeffer conjecture for systems of linear
  forms]
  Fix $m,n\in\NN$. Let
  \begin{equation*}
    P(\q)
    = \set{\p\in\ZZ^m : \gcd(p_i, \q) = 1,\, i = 1, \dots, m} 
  \end{equation*}
  Then
  \begin{equation*}
    \Leb(W_{n,m}^P(\psi)) =
    \begin{cases}
      0 &\textrm{if } \sum_{\q\in\ZZ^n\setminus\set{\bzero}} \parens*{\frac{\varphi(\gcd(\q))\psi(\q)}{\gcd(\q)}}^m  < \infty \\
      1 &\textrm{if } \sum_{\q\in\ZZ^n\setminus\set{\bzero}} \parens*{\frac{\varphi(\gcd(\q))\psi(\q)}{\gcd(\q)}}^m = \infty.
    \end{cases}
  \end{equation*}
\end{theorem*}

This theorem also follows easily now, by applying
\th\ref{thm:bootstrap11} with the results of Pollington--Vaughan and
Koukoulopoulos--Maynard as input. The relative well-spreadedness of
$P$ follows from~\cite[Lemma~5]{ramirez2023duffinschaeffer}. In fact,
the proofs of \th\ref{thm:bootstrap,thm:bootstrap11} are abstractions
of the arguments in~\cite{ramirez2023duffinschaeffer}.

In~\cite{ds_counterex}, Duffin--Schaeffer-style counterexamples are
presented, showing that the monotonicity condition cannot be removed
from the one-dimensional inhomogeneous Khintchine theorem, and an
inhomogeneous version of the Duffin--Schaeffer conjecture is
discussed. Specifically, it is the $m=n=1$ version of the above stated
theorem, but with
\begin{equation*}
  P(q)
  = \set{p + y : p\in\ZZ, \gcd(p, q) = 1,\, i = 1, \dots, m}
\end{equation*}
for every $q\in\Zpos$, where $y\in\RR$ is an arbitrarily fixed
inhomogeneous parameter. A version of that conjecture for systems of
linear forms would go as follows.

\begin{conjecture}[Inhomogeneous Duffin--Schaeffer conjecture for
  systems of linear forms]\th\label{conj:inhomdsc}
  Fix $m,n\in\NN$. Let $\y\in\RR^m$ and let
  \begin{equation*}
    P(\q)
    = \set{\p + \y : \p\in\ZZ^m,  \gcd(p_i, \q) = 1,\, i = 1, \dots, m}.
  \end{equation*}
  Then
  \begin{equation*}
    \Leb\parens*{W_{n,m}^P(\psi)} =
    \begin{cases}
      0 &\textrm{if } \sum_{\q\in\ZZ^n\setminus\set{\bzero}} \parens*{\frac{\varphi(\gcd(\q))\psi(\q)}{\gcd(\q)}}^m  < \infty \\
      1 &\textrm{if } \sum_{\q\in\ZZ^n\setminus\set{\bzero}} \parens*{\frac{\varphi(\gcd(\q))\psi(\q)}{\gcd(\q)}}^m = \infty.
    \end{cases}
  \end{equation*}
\end{conjecture}

There has not been much progress on \th~\ref{conj:inhomdsc} in the
$m=n=1$ setting, although some evidence in favor of it is presented
in~\cite{ds_counterex} and special cases of a weak version of it
(stated below as \th\ref{thm:chau}) have been established by Chow and
Technau. Part of what makes the problem difficult is that the
coprimality requirement no longer accomplishes what it accomplished in
the homogeneous setting. In the homogeneous setting, the coprimality
condition avoids the main issue underlying the Duffin--Schaeffer
counterexample, namely, that the measure sum $\sum\psi(q)$ appearing
in the non-reduced (Khintchine) setup may be extremely redundant due
to coincidences of the form $p_1/q_1 = p_2/q_2$ for $q_1\neq q_2$. In
the inhomogeneous setting, the coprimality condition does not directly
address the issue underlying the counterexamples
from~\cite{ds_counterex}, namely, that the measure sum $\sum\psi(q)$
appearing in the non-reduced (inhomogeneous Khintchine) setup may be
extremely redundant due to \emph{near} coincidences of the form
$(p_1 + y)/q_1 \approx (p_2 + y)/q_2$ for $q_1\neq q_2$. One could
take this as a sign that \th~\ref{conj:inhomdsc} may not be the
``correct'' inhomogeneous generalization of
Duffin--Schaeffer. (Another is proposed by
Chow--Technau~\cite[Question~1.24]{chow2023littlewood}.) Nevertheless,
it is open, and a solution one way or the other would be interesting.

Regarding the higher-dimensional cases of \th~\ref{conj:inhomdsc}, a
univariate version---meaning that $\psi$ only depends on
$\abs{\q}$---has been established for $n>2$ by Allen and
the author~\cite[Theorem~3]{allen2023inhomogeneous}, lending further
support for the plausibility of the conjecture.

In any case, \th\ref{thm:bootstrap11} allows one to prove the
following conditional result.

\begin{theorem}\th\label{thm:hypotheticalidsc}
  Fix $m\in\NN$ and $\y\in\RR^m$. If \th~\ref{conj:inhomdsc} holds
  with inhomogeneous parameter $\y$ and with $n=1$, then it holds with
  inhomogeneous parameter $\y$ for all $n\geq 1$. In particular, if
  the $m$-dimensional inhomogeneous Duffin--Schaeffer conjecture is
  true, then so is the $n$-by-$m$-dimensional inhomogeneous
  Duffin--Schaeffer conjecture.
\end{theorem}

\begin{proof}
  As usual, the convergence part is not at issue. Let
  $\psi:\ZZ^n\to\Rpos$ satisfy the divergence condition of the
  theorem. As we have done in \S\ref{sec:gall-1965-impl} and in the
  proof of \th~\ref{thm:multivariateAR}, we may focus on case where
  $\psi:\Zpos^n\to\Rpos$ and
  \begin{equation*}
    \sum_{\q\in\Zpos^n\setminus\set{\bzero}} \parens*{\frac{\varphi(\gcd(\q))\psi(\q)}{\gcd(\q)}}^m = \infty.
  \end{equation*}
  That is, $\psi$ is supported in $\Zpos^n$.

  If there exists $\q\in\Zpos^n$ with $\gcd(\q)=1$ such that
  \begin{equation}\label{eq:25}
    \sum_{d\geq 1} \parens*{\frac{\varphi(d)\psi(d\q)}{d}}^m = \infty,
  \end{equation}
  then by assumption $\Leb\parens{W_{1,m}^P(\psi_\q)}=1$, where
  $\psi_\q:\Zpos\to\Rpos$ by $\psi_\q(d) = \psi(d\q)$. But in this case
  \begin{equation*}
    T^{-1}\parens*{W_{1,m}^P(\psi_\q)} \subset W_{n,m}^P(\psi), 
  \end{equation*}
  where $T:\TT^{mn}\to \TT^m$ is the map $\X \mapsto \q\X \pmod
  1$. That map preserves measure, therefore,
  $\Leb(W_{n,m}^P(\psi))=1$.

  On the other hand, suppose there is no $\q\in\Zpos^n$ with $\gcd(\q)=1$ for
  which~(\ref{eq:25}) holds. Then for every $Q\geq 1$, 
  \begin{equation}\label{eq:26}
    \sum_{d\geq 1} \sum_{\substack{\gcd(\q)=1 \\ \abs{\q} < Q}}
    \parens*{\frac{\varphi(d)\psi(d\q)}{d}}^m  < \infty
  \end{equation}
  therefore
  \begin{equation*}
    \sum_{d\geq 1}\parens*{\frac{\varphi(d)\Psi_Q(d)}{d}}^m
    = \sum_{d\geq
    1}\parens*{\frac{\varphi(d)\Psi(d)}{d}}^m
    - \sum_{d\geq 1}\sum_{\substack{\gcd(\q)=1 \\ \abs{\q} <
    Q}}\parens*{\frac{\varphi(d)\psi(d\q)}{d}}^m = \infty
  \end{equation*}
  where $\Psi_Q$ is defined
  by~(\ref{eq:Psidef}). \th~\ref{prop:Psifinite} shows that we may
  assume that $\Psi_Q(d)$ is always finite. Now, by assumption,
  $\Leb(W_{1,m}^P(\Psi_Q)) = 1$. Since $Q\geq 1$ was arbitrary,
  \th\ref{thm:bootstrap11} gives $\Leb(W_{n,m}^P(\psi)) = 1$,
  finishing the proof.
\end{proof}

\subsection{Chow--Technau 2023 $\implies$ a version for dual approximation}

For many of the years during which the Duffin--Schaeffer conjecture
was open, there was a seemingly weaker related problem, Catlin's
conjecture~(1976, \cite{CatlinConj}), that also remained open and was
only established as a consequence of Duffin--Schaeffer once
Koukoulopoulos and Maynard proved it~\cite{KMDS}. The following is the
version of Catlin's conjecture for systems of linear forms, proved in
dimensions $m>1$ by Beresnevich--Velani~\cite{BVKG} and $m=1, n>1$ by
the author~\cite{ramirez2023duffinschaeffer}.

\begin{theorem*}[Catlin's conjecture for systems of linear forms]
  Fix $m,n\in\NN$. For $\q\in\ZZ^n$, let
  \begin{equation*}
    \Phi_m(\q) = \#\set{\p\in\ZZ^m : \abs{\p} \leq \abs{\q},
      \gcd(\p,\q)=1},
  \end{equation*}
  and $P(\q) = \ZZ$. Then
  \begin{equation*}
    \Leb(W_{n,m}^P(\psi)) =
    \begin{cases}
      0 &\textrm{if } \sum_{\q\in\ZZ^n\setminus\set{0}}\Phi_m(\q) \sup_{t\geq 1}\parens*{\frac{\psi(t\q)}{t\abs{\q}}}^m  < \infty \\
      1 &\textrm{if } \sum_{\q\in\ZZ^n\setminus\set{0}}\Phi_m(\q) \sup_{t\geq 1}\parens*{\frac{\psi(t\q)}{t\abs{\q}}}^m = \infty.
    \end{cases}
  \end{equation*}
\end{theorem*}

It is not clear what the appropriate version of Catlin's conjecture
should be in the inhomogeneous setting. In the classical statement,
the series has a clear purpose: it ensures that each rational number
$p/q$ contributes exactly once to the measure sum, with the maximal
measure of an approximation interval of which it is the center. The
series does something similar in higher dimensions $n,m\geq 1$. But if
an inhomogeneous parameter $\y\neq \bzero$ is introduced, then this
effect is lost.

On the other hand,
in~\cite{randomfractions,ramirez2023duffinschaeffer} it is shown that
in the cases where $m=1$ (the ``dual'' cases), the divergence part of
Catlin's conjecture for systems of linear forms is equivalent to the
following weak version of the Duffin--Schaeffer conjecture.

\begin{theorem*}[Weak dual Duffin--Schaeffer conjecture]
  Fix $n\in\NN$ and for every $\q\in\ZZ^n$ let $P(\q) = \ZZ$. Then
  \begin{equation*}
    \Leb(W_{n,1}^P(\psi)) = 1 \qquad\textrm{if}\qquad \sum_{\q\in\ZZ^n\setminus\set{\bzero}} \frac{\varphi(\gcd(\q))\psi(\q)}{\gcd(\q)} = \infty.
  \end{equation*}
\end{theorem*}

This theorem is a direct consequence of the $n$-by-$1$
Duffin--Schaeffer conjecture.  Of course, its statement can be made
inhomogeneous exactly as in \th~\ref{conj:inhomdsc}.

\begin{conjecture}[Weak inhomogeneous Duffin--Schaeffer conjecture for
  systems of linear forms]\th\label{conj:weakidsc}
  Fix $m,n\in\NN$ and $\y\in\RR^m$. For every $\q\in\ZZ^n$ let
  $P(\q) = \y + \ZZ^m$. Then
  \begin{equation*}
    \Leb(W_{n,m}^P(\psi)) = 1 \qquad\textrm{if}\qquad \sum_{\q\in\ZZ^n\setminus\set{\bzero}} \parens*{\frac{\varphi(\gcd(\q))\psi(\q)}{\gcd(\q)}}^m = \infty.
  \end{equation*}
\end{conjecture}

The set $P(\q)$ appearing in \th~\ref{conj:inhomdsc} is a subset of
the one appearing in \th~\ref{conj:weakidsc}, which is why
\th~\ref{conj:weakidsc} is referred to as weaker. But, of course, that
may not mean it is easier to prove.

Again, \th\ref{thm:bootstrap11} allows one to prove a conditional
result.

\begin{theorem}\th\label{thm:hypotheticalwidsc}
  Fix $m\in\NN$ and $\y\in\RR^m$. If \th~\ref{conj:weakidsc} holds
  with inhomogeneous parameter $\y$ and with $n=1$, then it holds with
  inhomogeneous parameter $\y$ for all $n\geq 1$. In particular, if
  the $m$-dimensional weak inhomogeneous Duffin--Schaeffer conjecture
  is true, then so is the $n$-by-$m$-dimensional weak inhomogeneous
  Duffin--Schaeffer conjecture.
\end{theorem}

\begin{proof}
  The proof is identical to that of \th~\ref{thm:hypotheticalidsc}.
\end{proof}

In~\cite{chow2023littlewood}, Chow and Technau prove
\th~\ref{conj:weakidsc} in dimension $m=n=1$ for a special class of
functions.

\begin{theorem}[Chow--Technau,~{\cite[Theorem~1.23]{chow2023littlewood}}]\th\label{thm:chau}
  Let $k\geq 1$. Fix $\bga=(\ga_1, \dots, \ga_k)\in\RR^k$ and
  $\gamma_1, \dots, \gamma_k\in\RR$. For $k=1$, suppose that
  $\alpha_1$ is an irrational non-Liouville number, and for $k\geq 2$
  assume
  \begin{equation*}
    \sup\set*{w>0 : \norm{q\ga_1}\cdot\norm{q\ga_2}\cdots\norm{q\ga_k} < q^{-w} \textrm{ for infinitely many } q\in\NN} < \frac{k}{k-1}.
  \end{equation*}
  Let $\bar\Psi:\NN\to\Rpos$ be a nonincreasing function such that
  $\sum\bar\Psi(q)(\log q)^k$ diverges, and let
  \begin{equation*}
    \Psi(q) = \frac{\bar\Psi(q)}{\norm{q\alpha_1 - \gamma_1}\cdots\norm{q\alpha_k - \gamma_k}}\qquad (q\in\NN). 
  \end{equation*}
  Then \th~\ref{conj:weakidsc} holds in dimension $m=n=1$ for $\Psi$. 
\end{theorem}

Using \th~\ref{thm:bootstrap11}, \th~\ref{thm:chau} can be
bootstrapped to prove a special case of the dual version of
\th~\ref{conj:weakidsc}.

\begin{theorem}[Special case of dual weak inhomogeneous
  Duffin--Schaeffer conjecture]\th\label{thm:chaubootstrap}
  Fix $n\geq 1$. Let $\bga$ and $\gamma_1, \dots, \gamma_k$ be as
  in \th~\ref{thm:chau}. Let $\bar\psi:\NN\to\Rpos$ be
  nonincreasing and such that
  \begin{equation}\label{eq:42}
    \sum_{\q\in\ZZ^n\setminus\set{\bzero}} \bar\psi(\abs{\q})(\log \gcd(\q))^k = \infty.
  \end{equation}
  Let $\psi:\ZZ^n\to\Rpos$ be defined by
  \begin{equation*}
    \psi(\q) = \frac{\bar\psi(\abs{\q})}{\norm{d\alpha_1 - \gamma_1}\cdots\norm{d\alpha_k - \gamma_k}}\qquad (d = \gcd(\q)). 
  \end{equation*}
  Then the $n$-by-$1$ \th~\ref{conj:weakidsc} holds for $\psi$.
\end{theorem}

\begin{proof}
  We lose no generality in considering only $\psi:\Zpos^n\to\Rpos$
  satisfying~\eqref{eq:42}.

  Let $y \in \RR$ and $P\equiv y + \ZZ$. The aim is to show that
  $\Leb\parens{W_{n,1}^P(\psi)}=1$. Suppose there exists
  $\q\in\Zpos^n$ with $\gcd(\q)=1$ such that
  \begin{equation}\label{eq:43}
    \sum_{d\geq 1} \bar\psi(\abs{d\q})(\log \gcd(d\q))^k = \sum_{d\geq 1} \bar\psi(\abs{d\q})(\log d)^k = \infty.
  \end{equation}
  Then with
  \begin{equation*}
    \bar\Psi_\q(d):=\bar\psi(\abs{d\q})\qquad\textrm{and}\qquad \tilde \Psi(d) := \psi(d\q)
  \end{equation*}
  playing the roles of $\bar\Psi$ and $\Psi$ from \th~\ref{thm:chau},
  that theorem gives $W_{1,1}^P(\tilde\Psi)=1$. But
  \begin{equation*}
    T^{-1}\parens*{W_{1,1}^P(\tilde\Psi)} \subset W_{n,1}^P(\psi)
  \end{equation*}
  where $T:\TT^{n}\to\TT$ is the measure-preserving $\X\mapsto \q\X$,
  so $\Leb\parens{W_{n,1}^P(\psi)}=1$ and the proof is finished.

  Assume, therefore, that there is no $\q\in\Zpos^n$ with $\gcd(\q)=1$
  for which~(\ref{eq:43}) holds. Consequently, for every $Q\geq 1$, we
  have
  \begin{equation}\label{eq:44}
    \sum_{\substack{\q\in\Zpos^n \\ \abs{\q/\gcd(\q)}\geq Q}} \bar\psi(\abs{\q})(\log \gcd(\q))^k = \infty.
  \end{equation}
  Define $\Psi$ and $\Psi_Q$ as in
  \th~\ref{thm:bootstrap,thm:bootstrap11}. Notice that
  \begin{align*}
    \Psi_Q(d)
    &= \sum_{\substack{\gcd(\q)=d \\ \abs{\q/\gcd(\q)}\geq Q}} \frac{\bar\psi(\abs{\q})}{\norm{d\alpha_1 - \gamma_1}\cdots\norm{d\alpha_k - \gamma_k}} \\
    &= \frac{\bar\Psi_Q(d)}{\norm{d\alpha_1 - \gamma_1}\cdots\norm{d\alpha_k - \gamma_k}}
  \end{align*}
  where
  \begin{equation*}
    \bar\Psi_Q(d) = \sum_{\substack{\gcd(\q)=1 \\ \abs{\q}\geq Q}} \bar\psi(\abs{d\q}).
  \end{equation*}
  Since $\bar\psi$ is nonincreasing, so is $\bar\Psi_Q(d)$. Observe that
  \begin{align*}
    \sum_{d\geq 1} \bar\Psi_Q(d)(\log d)^k &=     \sum_{d\geq 1}\sum_{\substack{\gcd(\q)=1 \\ \abs{\q}\geq Q}} (\log d)^k \bar\psi(\abs{d\q})\\
    &=     \sum_{\substack{\q\in\Zpos^n \\ \abs{\q/\gcd(\q)}\geq Q}}  \bar\psi(\abs{\q})(\log \gcd(\q))^k,
  \end{align*}
  which diverges, per~(\ref{eq:44}). Therefore, $\Psi_Q$ satisfies the
  conditions from \th~\ref{thm:chau}, hence
  $\Leb(W_{1,1}^P(\Psi_Q)) = 1$. Finally, since $Q\geq 1$ was
  arbitrary, \th~\ref{thm:bootstrap11} implies that
  $\Leb(W_{n,1}^P(\psi)) = 1$, as desired.
\end{proof}

\section{Examples}
\label{sec:examples}

In each of the applications in~\S\ref{sec:applications}, a
full-measure set in $\TT^m$ is bootstrapped to a full-measure set in
$\TT^{mn}$. Yet \th\ref{thm:bootstrap,thm:bootstrap11} do not require
full-measure in $\TT^m$ in order to apply. The next example presents
an application of \th~\ref{thm:bootstrap11} where a set of positive
and copositive measure in $\TT^m$ is bootstrapped to a full-measure
set in $\TT^{nm}$.

\begin{example}[Positive, copositive $\longrightarrow$
  full]\th\label{ex:postofull}
  Let $m>1$. For $d\geq 1$, let
  \begin{equation}\label{eq:51}
    P(d) = \set*{\p : 0\leq p_i \leq \frac{d}{2},\, i=1, \dots, m}.
  \end{equation}
  Let $\psi:\Zpos\to\Rpos$ be such that $\sum q^{n-1}\psi(q)^m$
  diverges, where $n>1$, and extend $\psi$ to $\Zpos^n$ by setting
  $\psi(\q) = \psi(\abs{\q})$. Let $Q\geq 1$. Then, as was shown
  in~\S\ref{sec:inhom-khintch-impl}, the series $\sum \Psi_Q(d)^m$
  also diverges, where $\Psi_Q$ is as in \th~\ref{thm:bootstrap11}. By
  Gallagher's extension of Khintchine's theorem,
  $\Leb\parens{W_{1,m}^\ZZ(\Psi_Q)}=1$, therefore,
  $\Leb\parens{W_{1,m}^P(\Psi_Q)}=2^{-m} > 0$. Since $P$ is
  well-spread in the sense of \th~\ref{def:well},
  \th~\ref{thm:bootstrap11} applies to give
  $\Leb\parens{W_{n,m}^P(\psi)}=1$.
\end{example}

The next is an example where \th\ref{thm:bootstrap11} does not apply
but \th\ref{thm:bootstrap} does. A positive, copositive measure set in
$\TT^m$ is bootstrapped to a positive, copositive measure set in
$\TT^{mn}$, illustrating the importance of the assumption on $\Psi_Q$
appearing in \th\ref{thm:bootstrap11}.

\begin{example}[Positive, copositive $\longrightarrow$ positive, copositive]\th\label{ex:postopos}
  Fix $m,n\in\NN$ and for $d\geq 1$, define $P(d)$ as
  in~(\ref{eq:51}). Let $\psi: \Zpos^n \to \Rpos$ be a function
  supported on vectors of the form $\q = (d, 0, 0, \dots, 0)$ and such
  that $\sum \psi(\q)^m$ diverges and $\psi(\q) = o(\abs{\q})$. Then
  the associated function $\Psi$ obviously satisfies
  $\sum_d \Psi(d)^m=\infty$, and Khintchine's theorem can be used to
  show that $W_{1,m}^P (\Psi) = 2^{-m}$. (It has full measure in
  $[0,1/2]^m$.) \th\ref{thm:bootstrap} gives
  $\Leb\parens{W_{n,m}^P(\psi)}>0$. But
  \begin{equation*}
    W_{n,m}^P(\psi) = W_{1,m}^P (\Psi) \times [0,1]^{(n-1)m},
  \end{equation*}
  and so $\Leb\parens{W_{n,m}^P(\psi)}=2^{-m}$.

  This example is easily generalized to situations where $\psi$ is
  supported on points $d\q'\, (d\geq 1)$ for some fixed $\q'$ with
  $\gcd(\q')=1$. This is the artificially $1$-by-$m$-dimensional
  construction mentioned after the statement of
  \th\ref{thm:bootstrap11}, and it is what the assumption on $\Psi_Q$
  avoids.
\end{example}

One may wonder whether \th~\ref{thm:bootstrap} might achieve a
full-measure set in $\TT^{mn}$ if it is further assumed that
$\Leb\parens{W_{1,m}^P(\Psi)}=1$. The next example shows that
\th~\ref{thm:bootstrap} does not necessarily bootstrap full measure to
full measure.

\begin{example}[Full $\longrightarrow$ positive,
  copositive]\th\label{ex:fulltopos}
  Let $(m,n)=(1,2)$. For $d\geq 1$ odd, let
  \begin{equation*}
    P(d) = \set*{p : 0\leq p \leq \frac{d}{2}}.
  \end{equation*}
  and for $d\geq 1$ even, let
  \begin{equation*}
    P(d) = \set*{p : \frac{d}{2}\leq p \leq d}.
  \end{equation*}
  Let $\psi:\Zpos\to\Rpos$ be nonincreasing and such that
  $\sum \psi(q)$ diverges. Extend $\psi$ to $\Zpos^2$ by setting
  \begin{equation*}
    \psi(\q) =
    \begin{cases}
      \psi(d) &\textrm{if } \q = \begin{pmatrix}d & 0\end{pmatrix}, d \textrm{ odd}\\
      \psi(d) &\textrm{if } \q = \begin{pmatrix}0 & d\end{pmatrix}, d \textrm{ even} \\
      0 &\textrm{otherwise.}
    \end{cases}
  \end{equation*}
  Notice that $\Psi(d) = \psi(d)$, where $\Psi$ is defined as in
  \th~\ref{thm:bootstrap}. By Khintchine's theorem,
  $\Leb\parens{W_{1,1}^\ZZ(\Psi)}=1$, and since $\Psi$ is
  nonincreasing it can be shown that
  $\Leb\parens{W_{1,1}^P(\Psi)}=1$. But $P$ is well-spread, so
  \th~\ref{thm:bootstrap} implies that
  $\Leb\parens{W_{2,1}^P(\psi)}>0$. Evidently, this cannot be improved
  to full measure. After all, if $\psi(q) = o(q)$, then
  \begin{equation*}
    W_{2,1}^P(\psi)\subset \parens*{[0,1/2]\times [0,1]} \cup\parens*{[0,1]\times [1/2, 1]} \subset \TT^2.
  \end{equation*}
  It has measure at most $3/4$.
\end{example}

\section{Measure-theoretic and geometric tools}
\label{sec:meas-theor-geom}

This section is a collection of important measure-theoretic and
geometric facts.

The first few illustrate a theme that is prevalent in metric number
theory: in order to show that the limsup set for a sequence
$(A_q)_{q=1}^\infty$ of subsets in a probability space has positive
measure, one should try to show that the sets $A_q$ are in some sense
independent. For example, the second Borel--Cantelli lemma says that
if they are pairwise independent and have diverging measure-sum, then
the limsup set has full measure. In practice, one must work with
weaker forms of independence. The ones that are most useful in metric
number theory are averaged forms of independence, like the ones
exhibited in the following propositions.

\begin{proposition}[Chung--Erd\H{o}s Lemma,~\cite{chungerdos}]\th\label{prop:erdoschung}
  If $(X, \mu)$ is a probability space and
  $(A_q)_{q \in \NN}\subset X$ is a sequence of measurable subsets
  such that $0 < \sum_{q} \mu(A_q) < \infty$, then
  \begin{equation*}
    \mu\parens*{\bigcup_{q=1}^\infty  A_q} \geq \frac{\parens*{\sum_{q=1}^\infty \mu(A_q)}^2}{\sum_{q,r=1}^\infty \mu(A_q\cap A_r)}.
  \end{equation*}
\end{proposition}

\begin{proof}
  This result is proved in~\cite{chungerdos} for finitely many sets
  $A_1, \dots, A_Q$, with the conclusion
  \begin{equation*}
    \mu\parens*{\bigcup_{q=1}^Q  A_q} \geq \frac{\parens*{\sum_{q=1}^Q \mu(A_q)}^2}{\sum_{q,r=1}^Q \mu(A_q\cap A_r)}.
  \end{equation*}
  This immediately implies
  \begin{equation*}
    \mu\parens*{\bigcup_{q=1}^\infty  A_q} \geq \frac{\parens*{\sum_{q=1}^Q \mu(A_q)}^2}{\sum_{q,r=1}^Q \mu(A_q\cap A_r)}.
  \end{equation*}
  for every $Q\geq 1$. The result stated here follows upon taking $Q\to\infty$. 
\end{proof}

The next result is an extension of the second Borel--Cantelli lemma
showing that a sequence of sets that are \emph{quasi-independent on
  average} and that have diverging measure sum must have a
positive-measure limsup set.
  
\begin{proposition}[Erd{\H o}s--Renyi,~\cite{ErdosRenyiBC}]\th\label{prop:erdosrenyi}
  Let $(X, \mu)$ be a probability space and $(A_k)_{k\geq 1}$ a
  sequence of measurable sets such that $\sum_{k\geq 1}\mu(A_k)$
  diverges. If there exists a constant $C\geq 0$ such that for
  infinitely many $D$,
\begin{equation}\label{eq:32}
  \sum_{k ,\ell \leq D} \mu(A_k\cap A_\ell)
  \leq C \parens*{\sum_{k \leq D} \mu(A_k)}^2
\end{equation}
then
\begin{equation*}
  \mu\parens*{\limsup_{k\to\infty} A_k} \geq \frac{1}{C}. 
\end{equation*}
\end{proposition}

\begin{remark*}
  Erd{\H o}s and Renyi proved this in~\cite{ErdosRenyiBC} with
  $C=1$. The proof of the more general statement above is found in many
  places, for example~\cite[Lemma~2.3]{Harman},~\cite[Chapter~1,
  Lemma~5]{Sprindzuk}.
\end{remark*}

The proofs of \th~\ref{thm:bootstrap,thm:bootstrap11} make direct use
of the following variant of \th~\ref{prop:erdosrenyi}.

\begin{proposition}\th\label{prop:qia}
Let $(X, \mu)$ be a probability space and $(A_{d,q})_{d,q\geq 1}$ a
collection of measurable sets such that for every $d\geq 1$,
\begin{equation}\label{eq:29}
  \sum_{q=1}^\infty \mu(A_{d,q}) < \infty \qquad\textrm{ while }\qquad
  \sum_{d=1}^\infty\sum_{q=1}^\infty \mu(A_{d,q}) =\infty. 
\end{equation}
If there exists a constant $C\geq 0$ such that for infinitely many
$D$, 
\begin{equation}\label{eq:30}
  \sum_{\substack{(k,q), (\ell,r) \\  k ,\ell \leq D}} \mu(A_{k,q}\cap A_{\ell,r})
  \leq C \parens*{\sum_{\substack{(d,q) \\ d \leq D}} \mu(A_{d,q})}^2
\end{equation}
then
\begin{equation*}
  \mu\parens*{\limsup_{d,q\to\infty} A_{d,q}} \geq \frac{1}{C}. 
\end{equation*}
\end{proposition}

\begin{proof}
  Let $\eps>0$. For each $d\geq 1$, let $q_d\geq 1$ be such that
  \begin{equation}\label{eq:31}
    \sum_{q=1}^{q_d}\mu(A_{d,q}) \geq (1-\eps)    \sum_{q=1}^{\infty} \mu(A_{d,q}).
  \end{equation}
  Let $(A_k)$ be  the sequence
  \begin{equation*}
    \set*{\set{A_{d,q}}_{q=1}^{q_d}}_{d=1}^\infty = \set*{A_{1,1}, \dots, A_{1, q_1}, A_{2,1}, A_{2,2}, \dots, A_{2,q_2},
    A_{3,1}, \dots}
  \end{equation*}
  For each $D\geq 1$ for which~(\ref{eq:30}) holds let
  $K = q_1 + q_2 + \dots + q_D$, so that
  \begin{equation*}
    \sum_{k,\ell \leq K}\mu(A_k\cap A_\ell) \leq C
    \parens*{\sum_{\substack{(d,q) \\ d \leq D}} \mu(A_{d,q})}^2\overset{(\ref{eq:31})}{\leq}
    \frac{C}{(1-\eps)^2} \parens*{\sum_{k \leq K} \mu(A_k)}^2. 
  \end{equation*}
  By \th\ref{prop:erdosrenyi},
  \begin{equation*}
    \mu\parens*{\limsup_{k\to\infty} A_k} \geq \frac{(1-\eps)^2}{C},
  \end{equation*}
  therefore
  \begin{equation*}
    \mu\parens*{\limsup_{d,q\to\infty} A_{d,q}} \geq \frac{(1-\eps)^2}{C}. 
  \end{equation*}
  Since $\eps>0$ was arbitrary, the same holds with $\eps=0$.
\end{proof}

Beresnevich and Velani have recently proved a collection of results
establishing partial converses to \th~\ref{prop:erdosrenyi}. The
following theorem states that if a sequence of balls has a
positive-measure limsup set, then it must contain a quasi-independent
subsequence.

\begin{theorem}[Beresnevich--Velani,~\cite{BVBC}]\th\label{thm:BV}
  Let $(X, \calA, \mu, d)$ be a metric measure space equipped with a
  doubling Borel probability measure $\mu$. Let $\set{B_i}_{i\in\NN}$ be a
  sequence of balls in $X$ with $r(B_i)\to 0$ as $i\to \infty$ and such that
  \begin{equation*}
    \exists a, b > 1 \quad\textrm{ such that}\quad \mu(a B_i) \leq b\mu(B_i) \quad\textrm{for all $i$ sufficiently large.}
  \end{equation*}
  Then
  \begin{equation*}
    \mu(\limsup_{i\to\infty} B_i) > 0
  \end{equation*}
  if and only if there exists a subsequence $\set{L_i}_{i\in\NN}$ of
  $\set{B_i}_{i\in\NN}$ and a constant $C>0$ such that
  \begin{equation*}
    \sum_{i=1}^\infty \mu(L_i) = \infty
  \end{equation*}
  and for infinitely many $Q\in \NN$
  \begin{equation*}
    \sum_{s,t = 1}^Q\mu(L_s\cap L_t) \leq C \parens*{\sum_{s=1}^Q \mu(L_s)}^2.
  \end{equation*}
  In the case where they both hold, one may take
  $C = K \mu(\limsup_{i\to\infty} B_i)^{-2}$, where $K>0$ is a
  constant depending only on $a,b$, and the doubling constant of the
  measure $\mu$.
\end{theorem}

\begin{remark*}
  This statement is a combination of~\cite[Theorem~3]{BVBC}, where the
  quasi-independence is asserted, and~\cite[Proposition~2]{BVBC},
  where the constant $C>0$ is specified. It is especially important to
  the proof of \th\ref{thm:bootstrap11} that the constant $C>0$ is understood.
\end{remark*}

The next lemma was proved by Cassels~\cite{Casselslemma} in the course
of proving his zero-one law for sets of the form
$W_{1,m}^\ZZ(\Psi)$. The lemma says that the measure of the limsup set
of a sequence of balls is not altered if the balls are uniformly
scaled. Cassels' lemma has enjoyed frequent use in metric Diophantine
approximation. Its role in~\S\ref{sec:1a} and~\S\ref{sec:1b} is in
allowing us to scale the function $\psi$ (and, correspondingly,
$\Psi$) until the balls making up certain approximation sets, to be
discussed in~\S\ref{sec:prop-appr-sets}, are disjoint.

\begin{lemma}[Cassels' lemma,~\cite{Casselslemma}]\th\label{lem:cassels}
  For each $k$, let $B_k$ be a ball in the torus $\TT^d$ having radius
  $c\psi_k$ where $\psi_k\geq 0$, $\psi_k\to 0\, (k\to\infty)$, and
  $c>0$ is a constant. Then $\Leb(\limsup_{k\to\infty} B_k)$ is
  independent of the value of the constant $c>0$.
\end{lemma}

A well-known consequence of the Lebesgue density theorem asserts that
if a set occupies a fixed positive proportion of every ball, then it
must have full measure. In the next lemma of Beresnevich, Dickinson,
and Velani, the requirement of a fixed positive proportion is
relaxed. The lemma plays a crucial role in the proof of
\th\ref{thm:bootstrap11}; it allows us to prove full measure of
$W_{n,m}^P(\psi)$ after having established positive measure on small
scales.

\begin{lemma}[Beresnevich--Dickinson--Velani,~{\cite[Lemma 6]{BDV}}]\th\label{lem:lebesguedensity}
  Let $(X,d)$ be a metric space with a finite measure $\mu$ such that
  every open set is $\mu$-measurable. Let $A$ be a Borel subset of $X$
  and let $f:\Rpos\to\Rpos$ be an increasing function with $f(x)\to 0$ as
  $x\to 0$. If for every open set $U\subset X$ we have
  \begin{equation*}
    \mu(A\cap U) \geq f(\mu(U)),
  \end{equation*}
  then $\mu(A) = \mu(X)$.
\end{lemma}

The next is Vitali's covering lemma. There are several versions, all
of which have a similar flavor: a collection of balls can be refined
to a subcollection of disjoint balls without losing too much measure
from the original collection. We state it for finitely many balls in
the torus. The lemma is used in the proofs of parts (b) of
\th~\ref{thm:bootstrap,thm:bootstrap11} in~\S\ref{sec:1b}
and~\S\ref{sec:2b}.

\begin{lemma}[Vitali's covering lemma]\th\label{lem:vitali}
  For every finite collection $B_1, \dots, B_k$ of balls in the torus
  $\TT^d$, there exists a subcollection $\tilde B_1, \dots, \tilde
  B_\ell$ such that
  \begin{equation*}
    \bigcup_{i=1}^k B_i \subset     \bigcup_{j=1}^\ell 3\bullet\tilde B_j,
  \end{equation*}
  (where $3\bullet B$ denotes the threefold concentric dilation of a
  ball $B$) and such that $\tilde B_i \cap \tilde B_j = \emptyset$ for all
  $i\neq j$. 
\end{lemma}

The next lemma is useful for understanding the behavior under scaling
of intersections of the approximation sets to be discussed
in~\S\ref{sec:prop-appr-sets}. In~\S\ref{sec:1a} it facilitates a
comparison between certain pairwise intersections arising in the
$n$-by-$m$ setting and the pairwise intersections in the $1$-by-$m$
setting.

\begin{lemma}[{\cite[Lemma~4]{ramirez2023duffinschaeffer}}]\th\label{lem:dilation}
  Suppose $I_1, I_2,\dots, I_r \subset \TT^m$ are disjoint balls and
  $J_1, J_2,\dots, J_s\subset \TT^m$ are disjoint balls. Then for any
  $0 < \Sigma \leq 1$,
  \begin{equation*}
    \Leb\parens*{\bigcup_{i=1}^r\Sigma \bullet I_i \cap \bigcup_{j=1}^s\Sigma\bullet J_j} \leq \Sigma^m \Leb\parens*{\bigcup_{i=1}^r I_i\cap\bigcup_{j=1}^s J_j},
  \end{equation*}
  where $\Sigma \bullet I_i$ denotes the concentric contraction of the
  ball $I_i$ by $\Sigma$, and similar for $\Sigma \bullet J_j$.
\end{lemma}

\begin{proof}
  The disjointness assumptions implies
  \begin{equation}\label{eq:45}
    \Leb\parens*{\bigcup_{i=1}^r I_i\cap\bigcup_{j=1}^s J_j} = \sum_{i,j}\Leb\parens*{I_i\cap J_j}
  \end{equation}
  and
  \begin{equation}\label{eq:46}
    \Leb\parens*{\bigcup_{i=1}^r \Sigma\bullet I_i\cap\bigcup_{j=1}^s \Sigma\bullet J_j} = \sum_{i,j}\Leb\parens*{\Sigma\bullet I_i\cap\Sigma\bullet J_j}.
  \end{equation}
  Let $1\leq i \leq r$ and $1\leq j \leq s$ and denote by $\bar I_i$
  and $\bar J_j$ the images of $I_i$ and $J_j$ under a scaling of the
  metric in $\TT^m$ by $\Sigma$, and let
  $\parens*{\bar\TT^m, \bar\Leb}$ denote the accordingly scaled
  measure space, namely, $\bar\TT^m = \TT^m$ as a set and
  $\bar\Leb = \Sigma^m\Leb$. Then we have
  \begin{equation*}
    \bar\Leb(\bar I_i) = \Leb(\Sigma\bullet I_i) = \Sigma^m \Leb(I_i) \quad\textrm{and}\quad\bar\Leb(\bar J_j) = \Leb(\Sigma\bullet J_j) = \Sigma^m\Leb(J_j).
  \end{equation*}
  Note that the centers of $\Sigma\bullet I_i$ and $\Sigma\bullet J_j$
  in $\TT^m$ are farther apart than the centers of $\bar I_i$ and
  $\bar J_j$ in $\bar\TT^m$. Therefore, 
  \begin{equation*}
    \Leb\parens*{\Sigma\bullet I_i\cap\Sigma\bullet J_j} \leq \bar\Leb\parens*{\bar I_i\cap\bar J_j} = \Sigma^m \Leb(I_i\cap J_j).
  \end{equation*}
  Putting this into~(\ref{eq:45}) and~(\ref{eq:46}) proves the lemma.  
\end{proof}

\section{Properties of approximation sets}
\label{sec:prop-appr-sets}

\th\ref{thm:bootstrap,thm:bootstrap11} concern limsups sets of the
form
\begin{equation*}
  W_{n,m}^P(\psi) = \limsup_{\abs{\q}\to\infty} A_{n,m}^P(\q, \psi(\q)),
\end{equation*}
where, for $\q\in\Zpos^n$, and $r \geq 0$,
\begin{equation*}
  A_{n,m}^P(\q, r) = \set*{\X \in \TT^{nm} : \q\X \in B(r) + P(\q)}.
\end{equation*}
For the proofs, it is essential that we understand the measures of the
sets $A_{n,m}^P(\q,\psi(\q))$ and their pairwise overlaps. This
section is a collection of lemmas about 
$A_{n,m}^P(\q,\psi(\q))$ and $A_{1,m}^P(d,\Psi(d))$.

The following lemma concerns the measures of
$A_{n,m}^P(\q,\psi(\q))$. Essentially, if $\psi(\q)$ is not too large,
then $A_{n,m}^P(\q,\psi(\q))$ can be viewed as an array of disjoint
parallelepipeds whose bases form a union of disjoint balls in
$\TT^{m}$.

\begin{lemma}\th\label{lem:measures}
  For each $d\geq 1$, let $P(d) \subset (\RR/d\ZZ)^m$ be a finite set,
  and for each $\q\in\Zpos^n$ let $P(\q)$ denote the lift to $\RR^m$
  of $P(\gcd(\q))$. If
  \begin{equation}\label{eq:11}
    r \leq \frac{1}{2}\min\set*{\abs{\p_1-\p_2} : \p_1,\p_2\in P(\q), \p_1\neq \p_2},
  \end{equation}
  then
  \begin{equation}\label{eq:6}
    \Leb\parens*{A_{n,m}^P(\q, r)} = \# P(\gcd(\q)) \parens*{\frac{2r}{\gcd(\q)}}^m. 
  \end{equation}
  In particular, if $P$ is relatively well-spread as in Definition~\ref{def:well} and
  \begin{equation}\label{eq:8}
    r \leq \frac{ad}{2\#(P(d))^{1/m}},
  \end{equation}
  where $a$ is as in~\eqref{eq:well-spread}, then
  \begin{equation}\label{eq:15}
    \Leb\parens*{A_{n,m}^{P'}(\q, r)} = \# P'(\gcd(\q)) \parens*{\frac{2r}{\gcd(\q)}}^m,
  \end{equation}
  hence
  \begin{equation}\label{eq:7}
    \Leb\parens*{A_{n,m}^{P'}(\q, r)} \asymp \# P(\gcd(\q)) \parens*{\frac{2r}{\gcd(\q)}}^m. 
  \end{equation}
\end{lemma}

\begin{proof}
  Denote $d = \gcd(\q)$ and let $\q'\in\Zpos^n$ be the unique
  primitive vector such that $\q = d\q'$. The mapping
  $T: \TT^{nm}\to \TT^m$ defined by $\X \mapsto \q'\X \pmod 1$ is
  Lebesgue measure-preserving, in the sense that for every measurable
  set $A\subset\TT^m$,
  \begin{equation}\label{eq:5}
    \Leb\parens*{T^{-1}(A)} = \Leb\parens*{A},
  \end{equation}
  where $\Leb$ is to be understood as $nm$-dimensional Lebesgue
  measure on the left-hand side, and $m$-dimensional Lebesgue measure
  on the right-hand side.

  Now, notice that
  \begin{equation}\label{eq:4}
    A_{n,m}^P(\q, r) = T^{-1}\parens*{A_{1,m}^P(d, r)}
  \end{equation}
  hence it suffices to compute the measure of $A_{1,m}^P(d, r)$. This
  set is a union of $\#P(d)$ balls in $\TT^m$ having radius $r/d$ and
  centers in the set
  \begin{equation*}
    \set*{\frac{\p}{d} : \p\in P(\gcd(\q))} + \ZZ^m \subset \TT^m. 
  \end{equation*}
  If~(\ref{eq:11}) holds, then those balls are disjoint, therefore
  \begin{equation*}
    \Leb\parens*{A_{1,m}^P(d, r)} = \#P(d) \parens*{\frac{2r}{d}}^m,
  \end{equation*}
  and~\eqref{eq:6} is proved after taking~\eqref{eq:5}
  and~\eqref{eq:4} into account.

  Suppose $P$ is relatively well-spread and~\eqref{eq:8} holds. Let
  $P'(d) \subset P(d)$ be as in Definition~\ref{def:well}. Now
  $A_{1,m}^{P'}(d, r)$ is a union of disjoint balls, and so
  \begin{equation*}
    \Leb\parens*{A_{1,m}^{P'}(d, r)} = \#P'(d) \parens*{\frac{2r}{d}}^m \geq c \#P(d) \parens*{\frac{2r}{d}}^m,
  \end{equation*}
  proving~\eqref{eq:15}. Since
  $A_{1,m}^{P'}(d, r)\subset A_{1,m}^P(d, r)$, it follows that
  \begin{equation*}
    c \#P(d) \parens*{\frac{2r}{d}}^m \leq \Leb\parens*{A_{1,m}^{P'}(d, r)} \leq  \#P(d) \parens*{\frac{2r}{d}}^m,
  \end{equation*}
  and~\eqref{eq:7} follows after applying~\eqref{eq:5}.   
\end{proof}

The next lemma implies that if $\psi(\q)$ is not too large, then the
sets $A_{n,m}^P(\q,\psi(\q))$ are pairwise independent for pairs
$\q_1, \q_2$ that are linearly independent. 

\begin{lemma}\th\label{lem:indA}
  Let $P(\gcd(\q)) \subset (\RR/\gcd(\q)\ZZ)^m$ be a finite set for
  all $\q\in\Zpos^n$. If $\q_1,\q_2\in\Zpos^n$ are linearly independent and each
  set $A_{1,m}^P(\gcd(\q_i), r_i),\, (i=1,2)$ is a union of disjoint balls, then
    \begin{equation*}
      \Leb\parens*{A_{n,m}^P(\q_1, r_1)\cap A_{n,m}^P(\q_2, r_2)} = \Leb\parens*{A_{n,m}^P(\q_1, r_1)}\Leb\parens*{A_{n,m}^P(\q_2, r_2)},
    \end{equation*}
    that is, the sets $A_{n,m}^P(\q_1, r_1)$ and
    $A_{n,m}^P(\q_2, r_2)$ are independent.
\end{lemma}

\begin{proof}
  The set  $A_{n,m}^{P}(\q_i, r_i)$ is a union of finitely many
  disjoint sets
  \begin{equation}\label{eq:2}
    A_{n,m}^{P}(\q_i, r_i) = \bigcup_{E\in\calE_i} E
  \end{equation}
  where each $E\in\calE_i$ is of the form
  \begin{equation*}
    E_{n,m}(\q_i, r_i, \bv_i) = \set*{\X \in M_{n\times m}(\RR) : \q_i\X \in B(\bv_i, r_i)}  + \ZZ^{nm} \subset \TT^{nm},
  \end{equation*}
  where $\bv_i = (v_1^{(i)}, \dots, v_m^{(i)})\subset \ZZ^m$ need not
  be specified in this argument. In turn, each such set is naturally a
  product of its projections,
  \begin{equation*}
    E_{n,m}(\q_i, r_i, \bv_i) = \prod_{j=1}^m E_{n,1}(\q_i, r_i, v_j^{(i)}). 
  \end{equation*}
  It is not hard to show (see for
  example~\cite[Lemma~1]{ramirez2023duffinschaeffer}) that the sets
  $E_{n,1}(\q_1)$ and $E_{n,1}(\q_2)$ are independent if $\q_1$ and
  $\q_2$ are linearly independent, hence the same holds for the sets
  $E_{n,m}(\q_i, r_i, \bv_i)\, (i=1,2)$. And since the
  unions~\eqref{eq:2} are disjoint, we may conclude that the sets
  $A_{n,m}^{P}(\q_i, r_i)\, (i=1,2)$ are independent.
\end{proof}

The next lemma concerns sets $E_{1,m}(q,r,\bv)$ having the form
discussed in the proof of \th~\ref{lem:indA}. With $n=1$, the set
$E_{1,m}(q,r,\bv)$ is a grid of balls in $\TT^m$. The lemma
observes that as $q\to\infty$, that grid of balls becomes uniformly
distributed in $\TT^m$.

\begin{lemma}\th\label{lem:equidistribution}
  Let $m\geq 1$. For every ball $W\subset \TT^m$, there exists
  $Q\geq 1$ such that for all $q\geq Q$, $r\geq 0$, and $\bv\in\RR^m$,
  \begin{equation}\label{eq:36}
    \Leb\parens*{E_{1,m}(q, r, \bv)\cap W}  \geq \frac{1}{2}
    \Leb\parens*{E_{1,m}(q, r, \bv)}\Leb\parens*{W}.
  \end{equation}
\end{lemma}

\begin{remark*}
  The precise value of the constant $1/2$ is not important. It can
  easily be increased to anything less than $1$, although such an
  improvement will not be needed here.
\end{remark*}

\begin{proof}
  If $r\geq 1/2$, then $E_{1,m}(q,r,\bv) = \TT^m$, so~(\ref{eq:36})
  obviously holds. Therefore, the lemma only needs to be proved for
  $0\leq r < 1/2$. In this case, $E_{1,m}(q,r,\bv)$ is a union of
  $q^m$ disjoint balls of radius $r/q$, so
  \begin{equation*}
    \Leb(E_{1,m}(q,r,\bv)) = (2r)^m.
  \end{equation*}
  Notice also that
  \begin{equation*}
    E_{1,m}(q, r, \bv)\cap W = E_{1,m}(q, r, \bzero)\cap (W+\bv')
  \end{equation*}
  for some $\bv'\in\RR^m$.

  Let
  \begin{equation*}
    \mu_q = \frac{1}{q^m}\sum_{\p\in (\ZZ/q\ZZ)^m}\delta_{\p/q}
  \end{equation*}
  be the measure uniformly supported on the rational points in $\TT^m$
  having denominator $q$. These measures equidistribute, that is, for
  every open ball $V\subset \TT^m$, we have
  \begin{equation}\label{eq:60}
    \mu_q(V) \to \Leb(V) \qquad\textrm{as} \qquad q\to \infty.
  \end{equation}
  One can prove this by taking an arbitrary ball $V\subset \TT^m$ of
  sidelength $\ell>0$ and computing
  \begin{align}
    \mu_q(V) &= \frac{1}{q^m}\sum_{\p\in (\ZZ/q\ZZ)^m}\delta_{\p/q}(V) \nonumber \\
             &= \frac{1}{q^m}\sum_{\p\in \ZZ^m}\delta_{\p}(qV),\label{eq:empirical}
  \end{align}
  where $qV$ is the image of $V\subset [0,1)^m\subset \RR^m$ under the
  $\times q $ map in $\RR^m$. Since $qV$ is an axis-parallel hypercube
  of sidelength $q\ell$, it is readily seen that
  \begin{align*}
    \sum_{\p\in \ZZ^m}\delta_{\p}(qV)
    &= \parens*{q\ell + O(1)}^m \\
    &= q^m\ell^m + O(q^{m-1}\ell^{m-1}),
  \end{align*}
  with an implicit constant depending only on $m$. Combining
  with~(\ref{eq:empirical}) gives
  \begin{align}
    \mu_q(V)
    &= \frac{1}{q^m}\parens*{q^m\ell^m + O(q^{m-1}\ell^{m-1})}\nonumber\\
    &= \ell^m + O\parens*{\frac{\ell^{m-1}}{q}} \nonumber \\
    &= \Leb(V) + O\parens*{\frac{\Leb(V)^{(m-1)/m}}{q}}.\label{eq:61}
  \end{align}
  Since for each ball $V$ the error term goes to $0$ as $q\to \infty$,
  the equidistribution~(\ref{eq:60}) follows. In fact,
  from~(\ref{eq:61}) one can see that the rate of equidistribution may
  depend on the volume of $V$, but not on the location of its
  center. That is, for every ball $V\subset \TT^m$ and $\eps>0$, there
  exists $q_0:=q_0(V,\eps)$ such that
  \begin{equation*}
    \abs{\mu_q(V+\bv) -\Leb(V+\bv)} < \eps
  \end{equation*}
  holds for all $\bv\in\RR^m$ and $q\geq q_0$.

  Let $W'$ be the concentric scaling of $W$ such that
  $\Leb(W') = \frac{1}{\sqrt{2}} \Leb(W)$. Then there exists $Q\geq 1$
  such that the following hold for all $q\geq Q$:
  \begin{enumerate}
  \item $\mu_q (W' + \bw) \geq \frac{1}{\sqrt{2}} \Leb (W'+\bw) = \frac{1}{2} \Leb(W+\bw)$ for all $\bw\in\RR^m$.
  \item Every ball of diameter $1/q$ centered in $W'$ is fully
    contained in $W$.
  \end{enumerate}
  The second point implies that for every $q\geq Q$, and every
  $0\leq r < 1/2$,
  \begin{equation*}
    \Leb\parens*{E_{1,m}(q,r,\bzero)\cap (W + \bv')} \geq (2r)^m\mu_q(W' + \bv') =  \Leb\parens*{E_{1,m}(q,r,\bzero)}\mu_q(W' + \bv'),
  \end{equation*}
  and the first point implies that 
  \begin{equation*}
    \Leb\parens*{E_{1,m}(q,r,\bzero)\cap (W+\bv')} \geq \frac{1}{2}\Leb\parens*{E_{1,m}(q,r,\bzero)}\Leb(W+\bv').
  \end{equation*}
  Equivalently,
  \begin{equation*}
    \Leb\parens*{E_{1,m}(q,r,\bv)\cap W} \geq \frac{1}{2}\Leb\parens*{E_{1,m}(q,r,\bv)}\Leb(W).
  \end{equation*}
  The lemma is proved.
\end{proof}

In the next lemma, it is observed that if $\abs{\q/\gcd(\q)}$ is
large, then a natural projection of $A_{n,m}^P(\q,\psi(\q))$ to
$\TT^m$ consists of sets to which \th~\ref{lem:equidistribution}
applies, and therefore the sets $A_{n,m}^P(\q,\psi(\q))$ inherit the
uniformity of that lemma. This is crucial for the proof of
\th\ref{thm:bootstrap11}.

\begin{lemma}\th\label{lem:regularity}
  For every open set $U\subset\TT^{nm}$ there exists $Q_0\geq 1$ such
  that the following holds.  For each $\q$, let
  $P(\gcd(\q))\subset(\RR/\gcd(\q)\ZZ)^m$ be a finite set. If
  $\abs{\q/\gcd(\q)} \geq Q_0$ and~(\ref{eq:11}) holds, then
  \begin{equation}\label{eq:48}
    \Leb\parens*{A_{n,m}^P(\q, r)\cap U} \geq \frac{1}{3} \Leb\parens*{A_{n,m}^P(\q, r)}\Leb\parens*{U}.
  \end{equation}
\end{lemma}

\begin{remark*}
  As in \th\ref{lem:equidistribution}, the precise value $1/3$ is
  not important.
\end{remark*}

\begin{proof}
  The set $A_{n,m}^P(\q,r)$ is the finite union
  \begin{equation}
    A_{n,m}^{P}(\q, r) = \bigcup_{E\in\calE} E,
  \end{equation}
  where each $E\in\calE$ is of the form
  \begin{equation*}
    E_{n,m}(\q, r, \bv) = \set*{\X \in M_{n\times m}(\RR) : \q\X \in B(\bv, r)}  + \ZZ^{nm} \subset \TT^{nm},
  \end{equation*}
  (for some $\bv:=\bv_E$) as in the proof of \th\ref{lem:indA}.  Since
  $r\geq 0$ satisfies~(\ref{eq:11}), the sets $E\in\calE$ are pairwise
  disjoint. It is therefore enough to prove this lemma with
  $E_{n,m}(\q, r, \bv)$ in place of $A_{n,m}^P(\q,r)$.

  Since $U$ is open, there are finitely many disjoint balls
  $V_1, V_2, \dots, V_L\subset U$ such that
  \begin{equation*}
    \sum_{i=1}^L \Leb(V_i) \geq \frac{2}{3}\Leb(U).
  \end{equation*}
  The lemma will be proved after establishing that there exists
  $Q_0\geq 1$ such that
  \begin{equation}\label{eq:35}
    \Leb\parens*{E_{n,m}(\q, r, \bv)\cap V} \geq \frac{1}{2}\Leb\parens*{E_{n,m}(\q, r, \bv)}\Leb\parens*{V}  \qquad (V= V_1, \dots, V_L)
  \end{equation}
  for all $\q\in\Zpos^n$ with $\abs{\q/\gcd(\q)} \geq Q_0$.

  Write
  \begin{equation}\label{eq:33}
    \TT^{mn} = \pi(\TT^{mn})\times\pi(\TT^{mn})^\perp 
    \cong \TT^m\times\TT^{mn-m}
  \end{equation}
  where $\pi$ is the orthogonal projection of $\TT^{nm}$ to the row
  corresponding to the coordinate in which $\abs{\q}$ is achieved, and
  $\pi(\TT^{mn})^\perp \cong \TT^{mn-n}$ is the orthogonal complement,
  so that every point in $\TT^{mn}$ has a unique expression as
  $\X = (\Y,\Z)$ in the product. Denote $\pi^\perp(\Y,\Z)=\Z$. The
  Lebesgue measure of $\TT^{mn}$ splits accordingly as
  \begin{equation*}
    \Leb_{nm} = \Leb_{m}\times \Leb_{mn-m},
  \end{equation*}
  which for the rest of this proof we will write as
  $\Leb = \Leb_1 \times \Leb_2$ to avoid cumbersome subscripts.  For a
  set $S\subset\TT^{mn}$, let
  \begin{equation*}
    S_\Z = \pi\parens*{S \cap (\TT^m \times \set{\Z})}
  \end{equation*}
  be the cross-section of $S$ through $\Z\in\pi(\TT^{mn})^\perp$, so
  that by Fubini's theorem, the measure of $S$ is
  \begin{equation*}
  \Leb(S) = \int \Leb_1(S_\Z)\, \operatorname{d}\Leb_2(\Z).
  \end{equation*}
  Then
  \begin{equation}\label{eq:34}
    \Leb\parens*{E_{n,m}(\q, r, \bv)\cap V} = \int \Leb_1\parens*{(E_{n,m}(\q, r, \bv)\cap V)_\Z}\,\diff\Leb_2(\Z).
  \end{equation}
  But for every $\Z\in\TT^{mn-m}$ there is some $\bv_\Z\in\RR^m$ such
  that
  \begin{align*}
    (E_{n,m}(\q, r, \bv)\cap V)_\Z
    &= E_{1,m}(\abs{\q/\gcd(\q)}, r/\gcd(\q), \bv_\Z)\cap V_\Z\\
    &=
    \begin{cases}
      E_{1,m}(\abs{\q/\gcd(\q)}, r/\gcd(\q), \bv_\Z)\cap \pi(V) &\textrm{if } \Z \in \pi^\perp(V)\\
      \emptyset &\textrm{if } \Z\notin \pi^\perp(V).
    \end{cases}
  \end{align*}
  The cross-sections are sets of the kind treated in
  \th\ref{lem:equidistribution}.

  Let $W_i = \pi(V_i)$ for $i=1, \dots, L$, and let $Q_i\geq 1$ be as
  in \th\ref{lem:equidistribution}. Put $Q_0 = \max_i Q_i$. Then, as
  long as $\abs{\q/\gcd(\q)} \geq Q_0$, we have
  \begin{align*}
    \Leb_1 \parens*{E_{1,m}(\abs{\q/\gcd(\q)}, r/\gcd(\q), \bv_\Z)\cap \pi(V)}
    &\geq \frac{1}{2}\Leb_1\parens*{E_{1,m}(\abs{\q/\gcd(\q)}, r/\gcd(\q), \bv_\Z)}\Leb_1(\pi(V))\\
    &= \frac{1}{2}\Leb\parens*{E_{n,m}(\q, r, \bv)}\Leb_1( \pi(V)).
  \end{align*}
  Returning to~(\ref{eq:34}),
  \begin{align*}
    \Leb\parens*{E_{n,m}(\q, r, \bv)\cap V}
    &= \int \Leb_1\parens*{(E_{n,m}(\q, r, \bv)\cap V)_\Z}\,\diff\Leb_2(\Z) \\
    &\geq \frac{1}{2} \Leb\parens*{E_{n,m}(\q, r, \bv)} \int \Leb_1( \pi(V))\bone_{\pi^\perp(V)}(\Z)\,\diff\Leb_2(\Z) \\
    &= \frac{1}{2} \Leb\parens*{E_{n,m}(\q, r, \bv)} \int \Leb_1(V_\Z)\,\diff\Leb_2(\Z) \\
    &= \frac{1}{2} \Leb\parens*{E_{n,m}(\q, r, \bv)} \Leb(V),
  \end{align*}
  giving~(\ref{eq:35}), which was the goal. 
\end{proof}

The next lemma is essentially a restatement of \th\ref{thm:BV}
(Beresnevich--Velani~\cite{BVBC}) that applies directly in the proofs
of \th\ref{thm:bootstrap,thm:bootstrap11}. Under the assumption that
$\Leb\parens{W_{1,m}^P(\Psi)}>0$, \th\ref{thm:BV} provides a
(sub)sequence of balls contained in the sets $A_{1,m}^P(d,\Psi(d))$
which is quasi-independent on average. The point of the following
lemma is to show, using \th\ref{thm:BV}, that there is a refinement
$R\subset P$ such that the sets $A_{1,m}^R(d,\Psi(d))$ themselves are
quasi-independent on average.

\begin{lemma}\th\label{lem:BVtranslation}
  Let $m\in\NN$. Suppose for every $d\geq 1$,
  $P(d)\subset (\RR/d\ZZ)^m$ is a finite set and for each
  $\q\in\Zpos^n$ let $P(\q)$ be the lift to $\RR^m$ of
  $P(\gcd(\q))$. Suppose $\Psi:\Zpos\to\Rpos$ with
  \begin{equation}\label{eq:54}
    \Psi(d) = o(d)\quad\textrm{and}\quad(\#P(d))^{1/m}\frac{\Psi(d)}{d}=O(1),
  \end{equation}
  and for each $d\geq 1$,
  \begin{equation}\label{eq:41}
    \Psi(d) < \frac{1}{2}\min\set*{\abs{\p_1
        - \p_2} : \p_1,\p_2 \in \bar P(d), \p_1\neq \p_2},
  \end{equation}
  where $\bar P(d)$ is the lift of $P(d)$ to $\RR^m$. If
  $\Leb (W_{1,m}^P(\Psi)) > 0$, then there exists a refinement
  $R\subset P$ such that
  \begin{equation}\label{eq:28}
    \sum_{d = 1}^\infty\Leb\parens*{A_{1,m}^{R}(d, \Psi(d))}  = \infty
  \end{equation}
  and such that for infinitely many $D\geq 1$,
  \begin{equation}\label{eq:37}
    \sum_{k,\ell = 1}^D\Leb\parens*{A_{1,m}^{R}(k, \Psi(k)) \cap A_{1,m}^{R}(\ell, \Psi(\ell))} \leq C \parens*{\sum_{d=1}^D \Leb\parens*{A_{1,m}^{R}(d, \Psi(d))}}^2
  \end{equation}
  where $C=K\Leb\parens{W_{1,m}^P(\Psi)}^{-2}$ and $K>0$ is an
  absolute constant.
\end{lemma}

\begin{proof}
  Each set $A_{1,m}^P(d, \Psi(d))$ is a union of $\#P(d)$ balls in
  $\TT^m$, which we may enumerate as 
  \begin{equation*}
    A_{1,m}^P(d, \Psi(d)) = \bigcup_{\p\in P(d)} B\parens*{\frac{\p}{d}, \frac{\Psi(d)}{d}} := \bigcup_{j=1}^{\#P(d)} B_{j}^{(d)},
  \end{equation*}
  concatenate as
  \begin{equation}\label{eq:39}
    \set*{\set*{B_j^{(d)}}_{j=1}^{\#P(d)}}_{d\geq 1},
  \end{equation}
  and relabel as $\set{B_i}_{i=1}^\infty$. The assumption
  $\Psi(d)=o(d)$ implies that the radii of the balls $B_i$ approach
  $0$ as $d\to\infty$. Since
  \begin{equation*}
    \Leb(\limsup_{i\to\infty} B_i) > 0,
  \end{equation*}
  \th~\ref{thm:BV} applies. It provides a subsequence $\set{L_i}$ of
  $\set{B_i}$ and a constant $C>0$ such that
  \begin{equation}\label{eq:27}
    \sum_{i=1}^\infty \Leb(L_i) = \infty
  \end{equation}
  and for infinitely many $Q\geq 1$, 
  \begin{equation}\label{eq:38}
    \sum_{s,t = 1}^Q\Leb(L_s\cap L_t) \leq \frac{C}{2}\parens*{\sum_{s=1}^Q \Leb(L_s)}^2.
  \end{equation}
  The subsequence $\set{L_i}$ corresponds naturally to a refinement
  $R(d)\subset P(d)$ $(d\geq 1)$ and a subsequence of~(\ref{eq:39}),
  which after reindexing can be written
  \begin{equation}\label{eq:40}
    \set*{\set*{B_j^{(d)}}_{j=1}^{\#R(d)}}_{d\geq 1}.
  \end{equation}
  The divergence~(\ref{eq:27}) is equivalent to~(\ref{eq:28}), so it
  is only left to establish~(\ref{eq:37}).

  Let $Q\geq 1$ be such that~(\ref{eq:38}) holds. There is a
  corresponding $D\geq 1$ such that
  \begin{equation*}
    \set*{\set*{B_j^{(d)}}_{j=1}^{\#R(d)}}_{d=1}^{D} \subset     \set*{L_i}_{i=1}^Q \subset 
    \set*{\set*{B_j^{(d)}}_{j=1}^{\#R(d)}}_{d=1}^{D+1}.
  \end{equation*}
  Then
  \begin{align*}
    \sum_{k,\ell = 1}^D\Leb\parens*{A_{1,m}^{R}(k, \Psi(k)) \cap A_{1,m}^{R}(\ell, \Psi(\ell))}
    &\leq \sum_{s,t = 1}^Q\Leb(L_s\cap L_t)\\
    &\leq \frac{C}{2}\parens*{\sum_{s=1}^Q \Leb(L_s)}^2\\
    &\leq \frac{C}{2}\parens*{\sum_{d=1}^{D+1} \Leb\parens*{A_{1,m}^{R}(d, \Psi(d))}}^2.
  \end{align*}
  The last inequality follows from the fact that for each $d\geq 1$,
  the balls $B_j^{(d)}$ from~(\ref{eq:40}) are disjoint, which in turn
  is a consequence of~(\ref{eq:41}). Finally, it is easily seen that
  \begin{align*}
    \sum_{d=1}^{D+1} \Leb\parens*{A_{1,m}^{R}(d, \Psi(d))}
    &= \sum_{d=1}^{D} \Leb\parens*{A_{1,m}^{R}(d, \Psi(d))} + \parens*{\frac{\Psi(D+1)}{D+1}}^m\#R(D+1) \\
    &\sim \sum_{d=1}^{D} \Leb\parens*{A_{1,m}^{R}(d, \Psi(d))}
  \end{align*}
  as $D\to\infty$, by~(\ref{eq:54}) and the now
  established~(\ref{eq:28}). Therefore, if $Q\geq 1$ is sufficiently
  large, then~(\ref{eq:37}) holds. Note that the constant $C>0$
  appearing in this lemma is twice the constant coming from
  \th\ref{thm:BV}, which takes the form as described.
\end{proof}

\section{Reductions}
\label{sec:reductions}

This section contains results showing that the conclusions of
\th\ref{thm:bootstrap,thm:bootstrap11} sometimes follow without having
to assume anything about $\Leb\parens{W_{1,m}^P(\Psi)}$ or
$\Leb(W_{1,m}^P(\Psi_Q))$. The task of proving
\th\ref{thm:bootstrap,thm:bootstrap11} then reduces to treating the
cases where the results of this section do not apply.

We first prove \th\ref{prop:Psifinite}, a result that appears
as~\cite[Chapter~1, Theorem~14]{Sprindzuk} in a form that can be
translated to the notation used here. We include a proof for
completeness.

\begin{proof}[Proof of \th\ref{prop:Psifinite}]
  Notice that if $\Psi_Q(d) = \infty$, then $\Psi_1(d)=\infty$. Let
\begin{equation*}
  b = \min\set*{\abs{\p_1
      - \p_2} : \p_1,\p_2 \in P(d), \p_1\neq \p_2}.
\end{equation*}
Suppose there are infinitely many $\q\in\Zpos^n$ with $\gcd(\q)=d$ for
which $\psi(\q) > b/2$, and denote by $I$ the set of such $\q$. For
each $\q\in I$ we have
\begin{equation}\label{eq:3}
  A_{n,m}^{P}(\q, \psi(\q)) \supset A_{n,m}^{P}(\q, b/2). 
\end{equation}
By \th\ref{lem:measures},
\begin{equation*}
  \Leb\parens*{A_{n,m}^{P}(\q, b/2)} = \frac{\# P(d)b^m}{d^m} > 0
\end{equation*}
for all $\q\in I$. Therefore,
\begin{equation*}
  \sum_{\q\in I}\Leb\parens*{A_{n,m}^{P}(\q, b/2)} = \infty. 
\end{equation*}
By \th\ref{lem:indA} the sets $A_{n,m}^{P}(\q, b/2)$ $(\q\in I)$ are
pairwise independent, therefore by the second Borel--Cantelli lemma,
\begin{equation*}
  \Leb\parens*{\limsup_{\substack{\abs{\q}\to\infty \\ \q\in I}} A_{n,m}^{P}(\q, b/2)} = 1,
\end{equation*}
hence, by~\eqref{eq:3}, $\Leb\parens{W_{n,m}^{P}(\psi)} = 1$.

If, on the other hand, $I$ is a finite set, then by $\Psi_1(d)=\infty$
we must have
\begin{equation*}
  \sum_{\substack{\gcd(\q)=d \\ \q\notin I}} \psi(\q)^m = \infty. 
\end{equation*}
But for each $\q\in\Zpos^n$ with $\gcd(\q)=d$ and $\q\notin I$ we have
\begin{equation*}
  \Leb\parens*{A_{n,m}^{P}(\q, \psi(\q))} = \frac{\#P(d)}{d^m}(2\psi(\q))^m,
\end{equation*}
therefore
\begin{equation*}
  \sum_{\substack{\gcd(\q)=d \\ \q\notin I}}\Leb\parens*{A_{n,m}^{P}(\q, \psi(\q))} = \infty. 
\end{equation*}
Again, by \th\ref{lem:indA} the sets $A_{n,m}^{P}(\q, \psi(\q))$
$(\gcd(\q)=d,\, \q\notin I)$ are pairwise independent, therefore,
\begin{equation*}
  \Leb\parens*{\limsup_{\substack{\abs{\q}\to\infty \\ \gcd(\q)=d \\ \q \notin I}} A_{n,m}^{P}(\q, \psi(\q))} = 1,
\end{equation*}
hence $\Leb\parens*{W_{n,m}^{P}(\psi)} = 1$.
\end{proof}

\begin{claim}\th\label{claim:psivsPalt}
  Suppose $P$ is relatively well-spread and let $a>0$ be as
  in~(\ref{eq:well-spread}). If there are infinitely many
  $\q\in\Zpos^n$ for which
  \begin{equation}\label{eq:12}
    \psi(\q)  \geq  \frac{a}{2}\parens*{\frac{\gcd(\q)}{(\#P(\gcd(\q)))^{1/m}}}
  \end{equation}
  holds, then $\Leb\parens*{W_{n,m}^P(\psi)} >0$. If for infinitely
  many $Q\geq 1$ there exists some $\q\in\Zpos^n$
  satisfying~(\ref{eq:12}) as well as $\abs{\q/\gcd(\q)}\geq Q$, then
  $\Leb\parens*{W_{n,m}^P(\psi)} =1$.
\end{claim}

\begin{proof}
  Let $P'$ be as in Definition~\ref{def:well}. For each $\q\in\Zpos^n$
  satisfying~\eqref{eq:12}, \th\ref{lem:measures} implies
  \begin{equation}\label{eq:13}
    \Leb\parens*{A_{n,m}^{P'}\parens*{\q,
        \psi(\q)}} \geq \Leb\parens*{A_{n,m}^{P'}\parens*{\q,
        \frac{a}{2}\parens*{\frac{\gcd(\q)}{(\#P(\gcd(\q)))^{1/m}}}}}
    \geq a^m c > 0,
  \end{equation}
  where $c>0$ is as in~(\ref{eq:udiscrete}). This immediately implies
  that $\Leb\parens*{W_{n,m}^P(\psi)}\geq a^m c > 0$, which proves the
  first part of the claim.

  For the second part of the claim, notice that we may form a sequence
  $\set{\q_k}_{k\geq 1}\subset\Zpos^n$ of vectors
  satisfying~(\ref{eq:12}), such that $\abs{\q_k/\gcd(\q_k)}$
  increases strictly. But each $\q_k/\gcd(\q_k)$ is primitive,
  so the vectors $\q_k$ are pairwise linearly independent. Therefore,
  by \th\ref{lem:indA}, the sets
  \begin{equation*}
    A_{n,m}^{P'}\parens*{\q_k, \frac{a}{2}\parens*{\frac{\gcd(\q_k)}{(\#P(\gcd(\q_k)))^{1/m}}}}
  \end{equation*}
  are pairwise independent. Since they all have measure
  $\geq a^m c>0$, their measure sum diverges and the second
  Borel--Cantelli lemma implies that
  \begin{equation*}
    \limsup_{k\to\infty} A_{n,m}^{P'}\parens*{\q_k, \frac{a}{2}\parens*{\frac{\gcd(\q_k)}{(\#P(\gcd(\q_k)))^{1/m}}}}
  \end{equation*}
  has full measure, hence $\Leb\parens*{W_{n,m}^P(\psi)} = 1$.
\end{proof}

\begin{claim}\th\label{claim:c}
  Assume $P$ is relatively well-spread and that
  \begin{equation}\label{eq:9}
    \psi(\q)  \leq  \frac{a}{2}\frac{\gcd(\q)}{(\#P(\gcd(\q)))^{1/m}} \qquad (\q\in\Zpos^n),
  \end{equation}
  where $a>0$ is as in~\eqref{eq:well-spread}, and $\Psi(d) < \infty$
  for all $d\geq 1$. If there exists $\gd>0$ such that
  \begin{equation*}
    \limsup_{d\to\infty}(\#P(d))^{1/m}\frac{\Psi(d)}{d} \geq \gd,
  \end{equation*}
  then $\Leb\parens*{W_{n,m}^P(\psi)} >0$. Furthermore,
    if for infinitely many $Q\geq 1$,
  \begin{equation}\label{eq:49}
    \limsup_{d\to\infty}(\#P(d))^{1/m}\frac{\Psi_Q(d)}{d} \geq \gd,
  \end{equation}
  then $\Leb\parens*{W_{n,m}^P(\psi)} = 1$.
\end{claim}

\begin{proof}
  Let $P'$ be as in Definition~\ref{def:well}. The
  assumption~\eqref{eq:9} guarantees that
  \begin{equation*}
    \Leb\parens*{A_{n,m}^{P'}(\q, \psi(\q))} = \#P'(\gcd(\q))\parens*{\frac{2\psi(\q)}{\gcd(\q)}}^{m} \geq c\#P(\gcd(\q))\parens*{\frac{2\psi(\q)}{\gcd(\q)}}^{m} 
  \end{equation*}
  for every $\q$ and therefore
  \begin{equation*}
    \sum_{\gcd(\q)=d}\Leb\parens*{A_{n,m}^{P'}(\q, \psi(\q))} \geq \sum_{\gcd(\q)=d}c\#P(d)\parens*{\frac{2\psi(\q)}{d}}^{m}  = c\#P(d)\parens*{\frac{2}{d}}^{m} \Psi(d)^m.
  \end{equation*}
  The assumption on $\Psi$ implies that there exists a number $M:=M_\gd >0$ and  an
  integer sequence $1\leq d_1 < d_2 < d_3 <\dots$ such that
  \begin{equation*}
    M \leq \sum_{\gcd(\q)=d_k}\Leb\parens*{A_{n,m}^{P'}(\q, \psi(\q))} < \infty\qquad (k \geq 1).
  \end{equation*}
  By \th\ref{prop:erdoschung}, for each $d \in \set{d_k}_{k\geq 1}$
  one has
  \begin{equation}\label{eq:14}
    \Leb\parens*{\bigcup_{\gcd(\q)=d}A_{n,m}^{P'}(\q, \psi(\q))}
    \geq \frac{\parens*{\sum_{\gcd(\q)=d} \Leb\parens*{A_{n,m}^{P'}(\q, \psi(\q))}}^2}{\sum_{\gcd(\q),\gcd(\br)=d}\Leb\parens*{A_{n,m}^{P'}(\q,\psi(\q))\cap A_{n,m}^{P'}(\br,\psi(\br)))}}
  \end{equation}
  and by \th\ref{lem:indA},
  \begin{multline*}
    \sum_{\gcd(\q),\gcd(\br)=d} \Leb\parens*{A_{n,m}^{P'}(\q,\psi(\q))\cap A_{n,m}^{P'}(\br,\psi(\br)))} \\
    = \sum_{\gcd(\q)=d} \Leb\parens*{A_{n,m}^{P'}(\q,\psi(\q))}  + \sum_{\substack{\gcd(\q),\gcd(\br)=d \\ \q \neq \br}} \Leb\parens*{A_{n,m}^{P'}(\q,\psi(\q))}\Leb\parens*{A_{n,m}^{P'}(\br,\psi(\br)))} \\
    \leq \sum_{\gcd(\q)=d} \Leb\parens*{A_{n,m}^{P'}(\q,\psi(\q))}  + \parens*{\sum_{\gcd(\q)=d} \Leb\parens*{A_{n,m}^{P'}(\q,\psi(\q))}}^2.
  \end{multline*}
  Putting this into~\eqref{eq:14} leads to
  \begin{align*}
    \Leb\parens*{\bigcup_{\gcd(\q)=d}A_{n,m}^{P'}(\q, \psi(\q))}
    &\geq \frac{\sum_{\gcd(\q)=d} \Leb\parens*{A_{n,m}^{P'}(\q, \psi(\q))}}{1 + \sum_{\gcd(\q)=d} \Leb\parens*{A_{n,m}^{P'}(\q, \psi(\q))}} \\
    &\geq \frac{M}{M + 1} > 0,
  \end{align*}
  for all $d\in\set{d_k}_{k\geq 1}$. Consequently, the limsup set
  associated to the sets $A_{n,m}^{P'}(\q, \psi(\q))$ with
  $\gcd(\q)\in\set{d_k}_{k\geq 1}$ has measure at least $M/(M+1)$, and
  so does $W_{n,m}^P(\psi)$. This proves the first part of the
  claim.
  
  For the second part, let $U\subset\TT^{nm}$ be an open set and let
  $Q_0\geq 1$ be such that~(\ref{eq:48}) holds for all
  $\abs{\q/\gcd(\q)}\geq Q_0$, as per
  \th~\ref{lem:regularity}. Suppose $Q\geq Q_0$ is such
  that~(\ref{eq:49}) holds and note that
  \begin{equation*}
    \Psi_Q(d) = \parens*{\sum_{\gcd(\q)=d}\psi_Q(\q)^m}^{1/m}
  \end{equation*}
  where
  \begin{equation*}
    \psi_Q(\q) =
    \begin{cases} \psi(\q) &\textrm{ if } \abs*{\frac{\q}{\gcd(\q)}}
                             \geq Q \\ 0 &\textrm{ otherwise }.
    \end{cases}
  \end{equation*}
  As in the first part, there is an integer sequence
  $1\leq d_1 < d_2 < \dots$ such that for all
  $d\in \set{d_k}_{k\geq 1}$,
  \begin{equation*}
    \sum_{\gcd(\q) = d} \Leb\parens*{A_{n,m}^{P'}(\q, \psi_Q(\q))} \geq M,
  \end{equation*}
  and taking \th~\ref{lem:regularity} into account,
  \begin{equation*}
    \sum_{\gcd(\q) = d} \Leb\parens*{A_{n,m}^{P'}(\q, \psi_Q(\q))\cap U} \geq \frac{1}{3}\Leb(U) M.
  \end{equation*}
  Now
  \begin{multline*}
    \Leb\parens*{\bigcup_{\gcd(\q)=d}A_{n,m}^{P'}(\q, \psi_Q(\q)) \cap U} \\
    \geq \frac{\parens*{\sum_{\gcd(\q)=d} \Leb\parens*{A_{n,m}^{P'}(\q, \psi_Q(\q))\cap U}}^2}{\sum_{\gcd(\q),\gcd(\br)=d}\Leb\parens*{A_{n,m}^{P'}(\q,\psi_Q(\q))\cap A_{n,m}^{P'}(\br,\psi_Q(\br)))\cap U}} \\
    \geq \frac{1}{9}\Leb(U)^2\frac{\parens*{\sum_{\gcd(\q)=d} \Leb\parens*{A_{n,m}^{P'}(\q, \psi_Q(\q))}}^2}{\sum_{\gcd(\q),\gcd(\br)=d}\Leb\parens*{A_{n,m}^{P'}(\q,\psi_Q(\q))\cap A_{n,m}^{P'}(\br,\psi_Q(\br)))}} \\
    \geq \frac{M}{9(M+1)}\Leb(U)^2,
  \end{multline*}
  after following the same reasoning as in the first part. This
  implies that
  \begin{equation*}
    \Leb\parens{W_{n,m}^{P'}(\psi)\cap U} \geq \frac{1}{9}\parens*{\frac{M}{M + 1}}\Leb(U)^2,
  \end{equation*}
  and therefore $\Leb\parens{W_{n,m}^P(\psi)}=1$ by an application of
  \th~\ref{lem:lebesguedensity}.
\end{proof}

\begin{claim}\th\label{claim:d}
  Assume $P$ is relatively uniformly discrete and that
  \begin{equation}
    \psi(\q)  \leq  \frac{b}{2}\qquad (\q\in\Zpos^n)
  \end{equation}
  where $b>0$ is as in~(\ref{eq:udiscrete}), and $\Psi(d) < \infty$
  for all $d\geq 1$. If there exists $\gd>0$ such that
  \begin{equation*}
    \limsup_{d\to\infty}\frac{\Psi(d)}{d} > \gd,
  \end{equation*}
  then $\Leb\parens*{W_{n,m}^P(\psi)} >0$. If, in addition,
  $\#P(d)\to \infty$ as $d\to\infty$, then
  $\Leb\parens*{W_{n,m}^P(\psi)} = 1$.
\end{claim}

\begin{proof}
  The assumption implies that there are infinitely many $d\geq 1$ such
  that
  \begin{equation}\label{eq:55}
    \Psi(d)^m  = \sum_{\gcd(\q)=d}\psi(\q)^m \geq (\delta d)^m.
  \end{equation}
  Let $P'$ be a refinement of $P$ as in \th~\ref{def:well}. Since
  $\psi(\q) \leq b/2$, the set $A_{1,m}^{P'}(\gcd(\q), \psi(\q))$ is a
  union of disjoint balls, and \th\ref{lem:measures} implies
  \begin{equation}\label{eq:16}
    \sum_{\gcd(\q)=d} \Leb\parens*{A_{n,m}^{P'}(\q, \psi(\q))} = \sum_{\gcd(\q)=d}(\#P'(d))\parens*{\frac{\psi(\q)}{d}}^m  \geq  \delta^m (\#P'(d)) 
  \end{equation}
  for each $d\geq 1$ satisfying~(\ref{eq:55}). By
  \th\ref{prop:erdoschung}, for each $d\geq 1$
  satisfying~\eqref{eq:16}, one has
  \begin{equation}\label{eq:17}
    \Leb\parens*{\bigcup_{\gcd(\q)=d}A_{n,m}^{P'}(\q, \psi(\q))}
    \geq \frac{\parens*{\sum_{\gcd(\q)=d} \Leb\parens*{A_{n,m}^{P'}(\q, \psi(\q))}}^2}{\sum_{\gcd(\q),\gcd(\br)=d}\Leb\parens*{A_{n,m}^{P'}(\q,\psi(\q))\cap A_{n,m}^{P'}(\br,\psi(\br)))}}.
  \end{equation}
  By \th\ref{lem:indA},
  \begin{multline*}
    \sum_{\gcd(\q),\gcd(\br)=d} \Leb\parens*{A_{n,m}^{P'}(\q,\psi(\q))\cap A_{n,m}^{P'}(\br,\psi(\br)))} \\
    = \sum_{\gcd(\q)=d} \Leb\parens*{A_{n,m}^{P'}(\q,\psi(\q))}  + \sum_{\substack{\gcd(\q),\gcd(\br)=d \\ \q \neq \br}} \Leb\parens*{A_{n,m}^{P'}(\q,\psi(\q))}\Leb\parens*{A_{n,m}^{P'}(\br,\psi(\br)))} \\
    \leq \sum_{\gcd(\q)=d} \Leb\parens*{A_{n,m}^{P'}(\q,\psi(\q))}  + \parens*{\sum_{\gcd(\q)=d} \Leb\parens*{A_{n,m}^{P'}(\q,\psi(\q))}}^2.
  \end{multline*}
  Putting this into~\eqref{eq:17} gives
  \begin{align}
    \Leb\parens*{\bigcup_{\gcd(\q)=d}A_{n,m}^{P'}(\q, \psi(\q))} 
    &\geq \frac{\sum_{\gcd(\q)=d} \Leb\parens*{A_{n,m}^{P'}(\q, \psi(\q))}}{1 + \sum_{\gcd(\q)=d} \Leb\parens*{A_{n,m}^{P'}(\q, \psi(\q))}} \nonumber \\
    &\geq \frac{\delta^m(\#P'(d))}{\delta^m(\#P'(d)) + 1} > 0,\label{eq:52}
  \end{align}
  for all $d\geq 1$ for which~\eqref{eq:16} holds. Since there are
  infinitely many such $d\geq 1$, and since for each such $d\geq 1$
  one has $\#P'(d)\geq 1$, it follows that in this case
  $\Leb\parens*{W_{n,m}(\psi)}>0$ and the theorem is proved. Under the
  additional assumption that $\#P(d)\to \infty,\, (d\to\infty)$, we
  have that $\#P'(d) \to \infty$, and the expression in~(\ref{eq:52})
  approaches $1$. From here, it follows that $\Leb\parens{W_{n,m}^P(\psi)}=1$.
\end{proof}

\section{Proof of \th\ref{thm:bootstrap}(a)}
\label{sec:1a}

By \th\ref{claim:psivsPalt}, it may be assumed that~\eqref{eq:9}
holds for all but finitely many $\q$.  No generality is lost in
further assuming that~\eqref{eq:9} holds for all
$\q\in\Zpos^n$. \th\ref{prop:Psifinite} and
\th\ref{claim:c} show that it is safe to assume that $\Psi(d)$ is
always finite and that
\begin{equation}\label{eq:47}
  (\#P(d))^{1/m}\frac{\Psi(d)}{d} \to 0 \qquad (d\to\infty).
\end{equation}
By scaling $\psi$ if necessary, it may be assumed that
\begin{equation}\label{eq:10}
  \Psi(d)  \leq  \frac{a}{2}\parens*{\frac{d}{(\#P(d))^{1/m}}}\qquad (d\geq 1)
\end{equation}
where $a>0$ is as in~\eqref{eq:well-spread}. Note that such a scaling
of $\Psi$ does not alter the measure of $W_{1,m}^{P}(\Psi)$, by
Cassels' lemma (\th\ref{lem:cassels}).

It is assumed that $\Leb\parens{W_{1,m}^P(\Psi)}>0$.  For each
$d\geq 1$, let $P'(d)$ be a refinement of $P(d)$ as in
\th\ref{def:well}, and then augment $P'(d)$ by including points of
$P(d)$ until $P'(d)$ is maximal with respect to the requirement that
the balls
\begin{equation*}
  B\parens*{\frac{\p}{d}, \frac{\Psi(d)}{d}}\qquad (\p \in P'(d))
\end{equation*}
be disjoint in $\TT^m$. Now $A_{1,m}^{P'}(d, \Psi(d))$ is a union of
$\#P'(d) \geq c \#P(d)$ disjoint balls in $\TT^m$ such that every ball
from $A_{1,m}^P(d, \Psi(d))$ is intersected. In particular,
\begin{equation*}
  A_{1,m}^P(d, \Psi(d)) \subset A_{1,m}^{P'}(d, 3\Psi(d)),
\end{equation*}
which implies that $\Leb\parens{W_{1,m}^{P'}(3\Psi)} \geq \Leb\parens{W_{1,m}^{P}(\Psi)}>0$. Then, by
Cassels' lemma (\th\ref{lem:cassels}), we have $\Leb\parens{W_{1,m}^{P'}(\Psi)} \geq \Leb\parens{W_{1,m}^{P}(\Psi)} >0.$

By~\eqref{eq:47} and~(\ref{eq:10}) we have~(\ref{eq:54})
and~(\ref{eq:41}), so~\th\ref{lem:BVtranslation} applies. It gives a
refinement $R(d)\subset P'(d)$ such that the sets
$A_{1,m}^{R}(d, \Psi(d))$ have a diverging measure sum and are
quasi-independent on average, that is, there exists a constant $C>0$
and infinitely many $D\geq 1$ such that
\begin{equation}\label{eq:20}
  \sum_{k,\ell = 1}^D\Leb\parens*{A_{1,m}^{R}(k, \Psi(k)) \cap A_{1,m}^{R}(\ell, \Psi(\ell))} \leq C \parens*{\sum_{d=1}^D \Leb\parens*{A_{1,m}^{R}(d, \Psi(d))}}^2.
\end{equation}
It is worth recalling that for each $d\geq 1$, the balls constituting
set $A_{1,m}^R(d,\Psi(d))$ are disjoint.

(We also point out that we may take
\begin{equation}\label{eq:btw}
 C = K\kappa^{-2} \qquad\textrm{where}\qquad   0 < \kappa \leq \Leb(W_{1,m}^P(\Psi)),
\end{equation}
and $K>0$ is a constant. This observation is not needed here, but will
be useful in the proof of \th\ref{thm:bootstrap11} in \S\ref{sec:2a}.)

The aim now is to show that the sets $A_{n,m}^{R}(\q, \psi(\q))$ are
quasi-independent on average. A bound is needed on
\begin{multline}\label{eq:18}
  \sum_{\gcd(\q), \gcd (\br) \leq D}  \Leb\parens*{A_{n,m}^{R}(\q, \psi(\q))\cap A_{n,m}^{R}(\br, \psi(\br))} \\
  = \sum_{\substack{\gcd(\q), \gcd (\br) \leq D \\ \q \not\parallel\br}}  \Leb\parens*{A_{n,m}^{R}(\q, \psi(\q))\cap A_{n,m}^{R}(\br, \psi(\br))}\\
  + \sum_{\substack{\gcd(\q), \gcd (\br) \leq D \\ \q \parallel \br}}  \Leb\parens*{A_{n,m}^{R}(\q, \psi(\q))\cap A_{n,m}^{R}(\br, \psi(\br))}.
\end{multline}
By \th\ref{lem:indA},
\begin{multline}\label{eq:19}
  \sum_{\substack{\gcd(\q), \gcd (\br) \leq D \\ \q \not\parallel\br}}  \Leb\parens*{A_{n,m}^{R}(\q, \psi(\q))\cap A_{n,m}^{R}(\br, \psi(\br))}\\
  = \sum_{\substack{\gcd(\q), \gcd (\br) \leq D \\ \q \not\parallel\br}}  \Leb\parens*{A_{n,m}^{R}(\q, \psi(\q))}\Leb\parens*{A_{n,m}^{R}(\br, \psi(\br))}\\
  \leq  \parens*{\sum_{\gcd(\q) \leq D}  \Leb\parens*{A_{n,m}^{R}(\q, \psi(\q))}}^2
\end{multline}
Evidently, only the last term~\eqref{eq:18} needs to be controlled. We
express it as
\begin{multline*}
  \sum_{\substack{\gcd(\q), \gcd (\br) \leq D \\ \q \parallel \br}}  \Leb\parens*{A_{n,m}^{R}(\q, \psi(\q))\cap A_{n,m}^{R}(\br, \psi(\br))}\\
  =
  \sum_{k,\ell = 1}^D\sum_{\substack{\gcd(\q) = k \\ \gcd (\br) = \ell \\ \q \parallel \br}} \Leb\parens*{A_{n,m}^{R}(\q, \psi(\q))\cap A_{n,m}^{R}(\br, \psi(\br))}.
\end{multline*}
Letting $\q' = \q/\gcd(\q) = \br/\gcd(\br)$, notice that
\begin{equation*}
  A_{n,m}^{R}(\q, \psi(\q))\cap A_{n,m}^{R}(\br, \psi(\br)) = T^{-1}\parens*{A_{1,m}^{R}(k, \psi(\q))\cap A_{1,m}^{R}(\ell, \psi(\br))},
\end{equation*}
where $T:\TT^{mn}\to\TT^m$ is the measure-preserving projection
$\X\mapsto\q'\X$. Hence, 
\begin{multline}\label{eq:21}
  \sum_{\substack{\gcd(\q), \gcd (\br) \leq D \\ \q \parallel \br}}  \Leb\parens*{A_{n,m}^{R}(\q, \psi(\q))\cap A_{n,m}^{R}(\br, \psi(\br))}\\
  =
  \sum_{k , \ell = 1}^D \sum_{\substack{\gcd(\q)=k \\ \gcd (\br) =\ell \\ \q \parallel \br}}  \Leb\parens*{A_{1,m}^{R}(k, \psi(\q))\cap A_{1,m}^{R}(\ell, \psi(\br))}\\
  \overset{\textrm{\th\ref{lem:dilation}}}{\leq}
  \sum_{k,\ell=1}^D \sum_{\substack{\gcd(\q)=k \\ \gcd (\br) =\ell \\ \q \parallel \br}}  \max\set*{\frac{\psi(\q)}{\Psi(k)}, \frac{\psi(\br)}{\Psi(\ell)}}^m\Leb\parens*{A_{1,m}^{R}(k, \Psi(k))\cap A_{1,m}^{R}(\ell, \Psi(\ell))} \\
  =
  \sum_{k,\ell=1}^D \Leb\parens*{A_{1,m}^{R}(k, \Psi(k))\cap A_{1,m}^{R}(\ell, \Psi(\ell))}\underbrace{\sum_{\substack{\gcd(\q)=k \\ \gcd (\br) =\ell \\ \q \parallel \br}}  \max\set*{\frac{\psi(\q)}{\Psi(k)}, \frac{\psi(\br)}{\Psi(\ell)}}^m}_{\leq 2}
\end{multline}
Combining~\eqref{eq:20},~\eqref{eq:18},~\eqref{eq:19},
and~\eqref{eq:21}, there are infinitely many $D\geq 1$ for which
\begin{multline}\label{eq:22}
  \sum_{\gcd(\q), \gcd (\br) \leq D}  \Leb\parens*{A_{n,m}^{R}(\q,
    \psi(\q))\cap A_{n,m}^{R}(\br, \psi(\br))} \\
    \leq \parens*{\sum_{\gcd(\q) \leq D}  \Leb\parens*{A_{n,m}^{R}(\q, \psi(\q))}}^2 + 2C \parens*{\sum_{d=1}^D \Leb\parens*{A_{1,m}^{R}(d, \Psi(d))}}^2.
\end{multline}
Finally, note that
\begin{align*}
  \sum_{d = 1}^D \Leb\parens*{A_{1,m}^{R}(d, \Psi(d))}
  &= \sum_{d = 1}^D \#R(d) \parens*{\frac{\Psi(d)}{d}}^m \\
  &= \sum_{d = 1}^D \sum_{\gcd(\q)=d} \#R(d) \parens*{\frac{\psi(\q)}{d}}^m \\
  &= \sum_{\gcd(\q) \leq D}  \Leb\parens*{A_{n,m}^{R}(\q, \psi(\q))}.
\end{align*}
Putting this into~\eqref{eq:22} shows that the sets
$A_{n,m}^{R}(\q, \psi(\q))$ are quasi-independent on average, that is,
there exists a constant $C'>0$ and infinitely many $D\geq 1$ such that
\begin{equation}\label{eq:23}
  \sum_{\gcd(\q), \gcd (\br) \leq D}  \Leb\parens*{A_{n,m}^{R}(\q, \psi(\q))\cap A_{n,m}^{R}(\br, \psi(\br))}
  \leq C' \parens*{\sum_{\gcd(\q) \leq D}  \Leb\parens*{A_{n,m}^{R}(\q, \psi(\q))}}^2. 
\end{equation}
Finally, by \th\ref{prop:qia}, their associated limsup set has
positive measure. This implies that
$\Leb\parens*{W_{n,m}^P(\psi)} > 0$. \qed

\begin{remark*}
  Evidently, we may take
  \begin{equation}
    \label{eq:btw2}
    C' = 1 + 2K\gk^{-2}
  \end{equation}
  for any $0 < \gk \leq \Leb(W_{1,m}^P(\Psi))$, where $K>0$ is a
  constant only depending on $\TT^m$ as in~(\ref{eq:btw}).
\end{remark*}

\section{Proof of \th\ref{thm:bootstrap}(b)}
\label{sec:1b}

Since it is assumed that $\psi$ is bounded, no generality is lost in
assuming that $\psi(\q) \leq b/2$ for all $\q$, where $b>0$ is as
in~(\ref{eq:udiscrete}), that is,
\begin{equation*}
  \abs*{\p_1 - \p_2} \geq b \qquad (d\geq 1, \p_1, \p_2 \in P'(d), \p_1\neq \p_2)
\end{equation*}
with $P'(d)\subset P(d)$ the refinement in
Definition~\ref{def:well}. After all, one can always scale $\psi$ by a
constant, resulting in a correspondingly scaled $\Psi$. By Cassels'
lemma (\th\ref{lem:cassels}), $\Leb(W_{1,m}^P(\Psi))$ is unchanged. By
\th\ref{prop:Psifinite} it may be assumed that $\Psi(d)$ is finite for
every $d\geq 1$, and by \th\ref{claim:d} it may be assumed that
$\Psi(d)=o(d)$. It is assumed that $\Leb\parens{W_{1,m}^P(\Psi)}>0$.

Since $A_{1,m}^P(d, \Psi(d))$ is a union of finitely many balls in the
torus $\TT^m$, Vitali's covering lemma (\th\ref{lem:vitali})
guarantees that that collection of balls contains a disjoint
subcollection whose three-fold dilates cover the original. In other
words, there exists a subset $Q(d)\subset P(d)$ such that
$A_{1,m}^{Q}(d, \Psi(d))$ is a union of finitely many \emph{disjoint}
balls and such that
\begin{equation*}
  \Leb\parens*{A_{1,m}^{Q}(d, \Psi(d))} \geq 3^{-m}\Leb\parens*{A_{1,m}^{P}(d, \Psi(d))}. 
\end{equation*}
In particular,
\begin{equation*}
  \sum_{d\geq 1} \Leb\parens*{A_{1,m}^{Q}(d, \Psi(d))} = \infty. 
\end{equation*}
Furthermore, since
\begin{equation*}
  A_{1,m}^{Q}(d, 3\Psi(d)) \supset A_{1,m}^{P}(d, \Psi(d)),
\end{equation*}
we have
\begin{equation*}
  \Leb\parens*{W_{1,m}^{Q}(3\Psi)}  \geq \Leb\parens*{W_{1,m}^{P}(\Psi)} >0,
\end{equation*}
hence, by Cassels' lemma (\th\ref{lem:cassels}), $\Leb\parens{W_{1,m}^{Q}(\Psi)} \geq \Leb\parens{W_{1,m}^{P}(\Psi)} >0$.

Since the balls constituting $A_{1,m}^Q(d,\Psi(d))$ are pairwise
disjoint, condition~(\ref{eq:41}) is satisfied by $Q$ and $\Psi$.
Now, by \th\ref{lem:BVtranslation}, we may further refine $Q$ to
$R(d)\subset Q(d)$ in such a way that the sets
$A_{1,m}^{R}(d, \Psi(d))$ have a diverging measure sum and are
quasi-independent on average, that is, there exists a constant $C>0$
and infinitely many $D\geq 1$ such that
\begin{equation}
  \sum_{k,\ell = 1}^D\Leb\parens*{A_{1,m}^{R}(k, \Psi(k)) \cap A_{1,m}^{R}(\ell, \Psi(\ell))} \leq C \parens*{\sum_{d=1}^D \Leb\parens*{A_{1,m}^{R}(d, \Psi(d))}}^2.
\end{equation}
The aim now is to show that the sets $A_{n,m}^{R}(\q, \psi(\q))$ are
quasi-independent on average. It is established exactly as in the
Proof of \th\ref{thm:bootstrap}(a) from~\eqref{eq:18} on,
verbatim. \qed

\section{Proof of \th\ref{thm:bootstrap11}(a)}
\label{sec:2a}

By \th\ref{prop:Psifinite} we may assume that $\Psi(d)$ and
$\Psi_Q(d)$ are always finite. By \th~\ref{claim:psivsPalt}, it may be
assumed that for $Q_0 \geq 1$ sufficiently large~(\ref{eq:9}) holds
for all but finitely many $\q\in\Zpos^n$ having
$\abs{\q/\gcd(\q)}\geq Q_0$. Select such a $Q_0\geq 1$, and note that
we lose no generality in assuming that~(\ref{eq:9}) is satisfied by
\emph{all} $\q\in\Zpos^n$ having $\abs{\q/\gcd(\q)}\geq Q_0$. (This
entails altering $\psi(\q)$ for at most finitely many $\q$.)

Let $U\subset\TT^{mn}$ be an open set and, if necessary, enlarge
$Q_0\geq 1$ so that it is large enough that \th\ref{lem:regularity}
applies. By \th\ref{claim:c}, we may further enlarge $Q_0$ until it is
large enough that for all $Q\geq Q_0$ there exists $d_Q\geq 1$ such that
\begin{equation}\label{eq:50}
  \Psi_Q(d) \leq \frac{a}{2}\parens*{\frac{d}{(\#P(d))^{1/m}}}\qquad
  (d\geq d_Q),
\end{equation}
where $a>0$ is as in~(\ref{eq:well-spread}).

Define $\psi_Q$ by
\begin{equation}\label{eq:psi_Q}
  \psi_Q(\q) =
  \begin{cases} \psi(\q) &\textrm{ if } \abs*{\frac{\q}{\gcd(\q)}}
\geq Q \\ 0 &\textrm{ otherwise }.
  \end{cases}
\end{equation}
Then
\begin{equation}\label{eq:Psi_Q}
  \Psi_Q(d) =
  \parens*{\sum_{\gcd(\q)=d}\psi_Q(\q)^m}^{1/m}.
\end{equation}
It has been assumed that
$\Leb\parens{W_{1,m}^P(\Psi_Q)} \geq \gk > 0$, where $\gk>0$ is a
constant. By~(\ref{eq:50}) and the assumption that $\#P(d)\to \infty$,
we have $\Psi_Q(d)/d \to 0$. We are therefore in a position to follow
the proof of \th\ref{thm:bootstrap}(a) in \S\ref{sec:1a}. It leads to
a refinement $R\subset P'\subset P$ such that~(\ref{eq:23}) holds for
the sets $A_{n,m}^{R}(\q, \psi_Q(\q))$ and such that their measure sum
diverges. Specifically, there are infinitely many $D\geq 1$ for which
\begin{equation}\label{eq:24}
  \sum_{\gcd(\q), \gcd (\br) \leq D}  \Leb\parens*{A_{n,m}^{R}(\q, \psi_Q(\q))\cap A_{n,m}^{R}(\br, \psi_Q(\br))}
  \leq C' \parens*{\sum_{\gcd(\q)
      \leq D}
    \Leb\parens*{A_{n,m}^{R}(\q,
      \psi_Q(\q))}}^2,
\end{equation}
where $C' = 1 + K\gk^{-2}$, with $K>0$ a constant, as remarked
in~(\ref{eq:btw2}). By \th\ref{lem:regularity},
\begin{equation*}
  \Leb\parens*{A_{n,m}^{R}(\q, \psi_Q(\q))\cap U} \geq \frac{1}{3} \Leb\parens*{A_{n,m}^{R}(\q, \psi_Q(\q))}\Leb\parens*{U}
\end{equation*}
for all $\q\in\Zpos^n$. Therefore, there are infinitely many $D\geq 1$
for which
\begin{multline*}
  \sum_{\gcd(\q), \gcd (\br) \leq D}  \Leb\parens*{A_{n,m}^{R}(\q, \psi_Q(\q)) \cap
                                        A_{n,m}^{R}(\br,
                                        \psi_Q(\br))\cap U} \\
                                      \leq \sum_{\gcd(\q), \gcd (\br) \leq D}  \Leb\parens*{A_{n,m}^{R}(\q, \psi_Q(\q))\cap A_{n,m}^{R}(\br, \psi_Q(\br))} \\
                                      \overset{(\ref{eq:24})}{\leq} C' \parens*{\sum_{\gcd(\q)
                                        \leq D}
                                        \Leb\parens*{A_{n,m}^{R}(\q,
                                        \psi_Q(\q))}}^2 \\
                                      \leq \frac{9C'}{\Leb(U)^2}\parens*{\sum_{\gcd(\q)
                                        \leq D}
                                        \Leb\parens*{A_{n,m}^{R}(\q,
                                        \psi_Q(\q))\cap U}}^2.
\end{multline*}
By \th\ref{prop:qia},
\begin{equation*}
  \Leb\parens*{W_{n,m}^{R}(\psi_Q)\cap
    U} = \Leb\parens*{\limsup_{\abs{\q}\to\infty}A_{n,m}^{R}(\q,\psi_Q(\q))\cap
  U} \geq \frac{\Leb(U)^2}{9C'},
\end{equation*}
and since
$W_{n,m}^{R}(\psi_Q)\subset
W_{n,m}^{P}(\psi)$, it follows that
\begin{equation*}
  \Leb\parens*{W_{n,m}^P(\psi)\cap
    U} \geq \frac{\Leb(U)^2}{9C'}.
\end{equation*}
Finally, \th\ref{lem:lebesguedensity} implies that
$\Leb\parens*{W_{n,m}^P(\psi)}=1$, completing the proof. \qed

\section{Proof of \th\ref{thm:bootstrap11}(b)}
\label{sec:2b}

As in the proof of \th\ref{thm:bootstrap}(b) in \S\ref{sec:1b},
it may be assumed that $\psi(\q) \leq b/2$ for all $\q$, where
\begin{equation*}
  \abs*{\p_1 - \p_2} \geq b \qquad (d\geq 1, \p_1, \p_2 \in P(d), \p_1\neq \p_2),
\end{equation*}
and, by \th\ref{prop:Psifinite}, that $\Psi(d)$ is finite for every
$d\geq 1$. In fact, \th\ref{claim:d} implies that $\Psi(d)=o(d)$ may
be assumed. It is further assumed that
$\Leb\parens*{W_{1,m}^P(\Psi)}>0$. 

Let $U\subset\TT^{nm}$ be an open set, and $Q\geq 1$ be as in
\th\ref{lem:regularity}. Define $\psi_Q$ as in~(\ref{eq:psi_Q}) so
that~(\ref{eq:Psi_Q}) holds. Note that $\psi_Q(\q) \leq b/2$ for all
$\q\in\Zpos^n$, and that $\Psi_Q(d)<\infty$ for all $d\geq 1$, and
that $\Psi_Q(d) = o(d)$. It is further assumed that
$\inf_Q\Leb\parens{W_{1,m}^P(\Psi_Q)} > 0$. The arguments in
\S\ref{sec:1b} show that there is a refinement $R\subset P$ such that
the sets $A_{n,m}^{R}(\q, \psi_Q(\q))$ are quasi-independent on
average and have diverging measure sum. The rest of this proof
proceeds exactly as in the proof of \th\ref{thm:bootstrap11}(a),
starting from~(\ref{eq:24}). \qed

\subsection*{Acknowledgments}

I thank Victor Beresnevich, Sanju Velani, and Manuel Hauke for
valuable discussions about this work, and for hosting me at the
University of York in June 2024, during which time I polished a second
draft. I also thank the referee for a careful reading and helpful
comments.

\bibliographystyle{plain}


\vfill

Wesleyan University

Department of Mathematics and Computer Science

Middletown, CT 06459

United States

\texttt{framirez@wesleyan.edu}

\end{document}